\documentclass[smallextended]{svjour3}       
\usepackage{amsmath,  amsfonts}
\usepackage{ amssymb }
\usepackage{nicematrix}
\renewcommand{\qed}{$\hfill\square$}

\usepackage{color}

\usepackage{cite}
\usepackage{indentfirst}
\usepackage{longtable}
\usepackage[title]{appendix}
\usepackage[linesnumbered,ruled,vlined]{algorithm2e}
\usepackage{bm}
\usepackage[hidelinks]{hyperref}
\usepackage{multirow}
\usepackage[tight]{subfigure}
\usepackage{listings}
\usepackage{xcolor}
\usepackage{subfigure}

\newcommand{\Bcal}{\mathcal B}

\newcommand{\Ebb}{\mathbb E}

\newcommand{\Ncal}{\mathcal N}

\newcommand{\R}{\mathbb R}

\newcommand{\Pbb}{\mathbb P}

\newcommand{\B}{\mathcal B}

\newcommand{\ex}[2]{\mathbb{E}_{#1}\left[#2\right]}

\newcommand{\real}{\mathbb{R}} 
\newcommand{\prn}[1]{\left({#1}\right)} 
\newcommand{\norm}[1]{\left\|{#1}\right\|} 
\newcommand{\abs}[1]{\left|{#1}\right|} 

\newcommand{\st}{\text{s.t.  }}

\newcommand\ldq\textquotedblleft
\newcommand\rdq\textquotedblright{}
\newcommand\mb\mathbb
\newcommand\mf\mathbf
\newcommand\tf\textbf

\DeclareMathOperator{\argmin}{arg~min}

\newtheorem{assumption}{Assumption}

\definecolor{mygreen}{rgb}{0,0.6,0}
\definecolor{mygray}{rgb}{0.5,0.5,0.5}
\definecolor{mymauve}{rgb}{0.58,0,0.82}
\title{\textbf{\showtitle}}
\author{\showauthor}

\newcommand{\bgap}{\mathcal{G}}
\newcommand{\xopt}{\mathcal X^*}
\newcommand{\KL}[2]{\mathrm{KL}\prn{{#1}\,\|\,{#2}}}

\usepackage{bbm}
\usepackage{graphicx}
\usepackage{multirow, booktabs}
\usepackage{multicol}

\newcommand{\be}{\begin{equation}}
\newcommand{\ee}{\end{equation}}

\renewenvironment{proof}{{ \it Proof:~}}{\hfill \qed\par}

\usepackage{xcolor}

\title{A Monte Carlo Policy Gradient Method with   Local Search  for Binary Optimization
}
\titlerunning{Monte Carlo Policy Gradient Method for Binary Optimization} 
\author{Cheng Chen \and Ruitao Chen \and Tianyou Li \and Ruichen Ao \and Zaiwen Wen}
\institute{Cheng Chen \at
              Academy for Advanced Interdisciplinary Studies, Peking University, Beijing, CHINA\\
              \email{chen1999@pku.edu.cn}           
            \and
           Ruitao Chen, Tianyou Li, Ruichen Ao\at
            School of Mathematical Sciences, Peking University, Beijing, CHINA\\
            \email{\{chenruitao,tianyouli,archer\_arc\}@stu.pku.edu.cn}
           \and
           Zaiwen Wen \at
            Beijing International Center for Mathematical Research, Peking University, Beijing, CHINA\\
            \email{wenzw@pku.edu.cn}
}

\begin{document}

\maketitle

\begin{abstract}
    Binary optimization has a wide range of applications in combinatorial optimization problems such as MaxCut, MIMO detection, and MaxSAT. However, these problems are typically NP-hard due to the binary constraints. We develop a novel probabilistic model to sample the binary solution according to a parameterized policy distribution. Specifically,   minimizing the KL divergence between the parameterized policy distribution and the Gibbs distributions of the function value leads to a stochastic optimization problem whose  policy gradient can be derived explicitly similar to reinforcement learning.  For coherent exploration in discrete spaces, parallel Markov Chain Monte Carlo (MCMC) methods are employed to sample from the policy distribution with diversity and approximate the  gradient efficiently. We further develop a filter  scheme to replace the original objective function by the one with the local search technique  to  broaden the horizon of the function landscape. Convergence to stationary points in expectation of the policy gradient method is established based on the concentration inequality  for   MCMC. Numerical results show that this  framework is very promising to provide near-optimal solutions for quite a few  binary optimization problems.
\end{abstract}

\section{Introduction}
In this paper, we consider the following binary optimization problem:
\begin{equation}
    \label{prob:intro}
    \min_x \quad  f(x),\quad
    \st \quad  x\in \Bcal_n,
\end{equation}
where $\Bcal_n = \{-1,1\}^n$ represents the binary set of dimension $n$ and $f$ can be any function defined on $\Bcal_n$. Binary optimization \eqref{prob:intro} is a classical form of optimization with a wide range of practical applications, including the MaxCut problem, the quadratic 0-1 programming problem, the Cheeger cut problem, the MaxSAT problem, and the multiple-input multiple-output (MIMO) detection problem. Obtaining the optimal solution for these problems is typically NP-hard due to the nature of binary constraints. Therefore, exact methods, such as branch-and-bound \cite{land1960automatic} and cutting plane \cite{gilmore1961linear} designed for linear integer programming, hardly work for general large-scale binary problems. Moreover, the objective function $f$ in \eqref{prob:intro} may be quite complicated, adding difficulties in finding the global optimum.

Approximate algorithms based on relaxation sacrifice optimality to find high-quality solutions. For binary quadratic problems, semidefinite programming (SDP) provides a relaxation of the primal problem and achieves near  optimal solutions with the particular rounding scheme\cite{wai2011cheap, burer2001projected}. This becomes a paradigm to solve the binary optimization problem with quadratic objective function. MADAM solves MaxCut based on SDP and ADMM \cite{hrga2021madam}. EMADM provides an entropy regularized splitting model for the SDP relaxation on the Riemannian manifold \cite{liu2022entropy}. BiqCrunch incorporates an efficient SDP-based bounding procedure into the branch-and-bound method\cite{krislock2017biqcrunch}.
Solving the dual problem of SDP relaxation is also utilized for the MaxCut problem \cite{rendl2010solving,krislock2014improved}, which can be further efficiently solved by a customized coordinate ascent algorithm \cite{buchheim2019sdp}.

Heuristic methods is the another way to provide a good solution for the problem within a reasonable time frame, but do not have a reliable theoretical guarantees in most cases. One of the classical heuristic methods is simulated annealing, which is developed to find a good solution by iteratively swapping variables between feasible and infeasible points with a probability that depends on the measurement called ``temperature" \cite{eglese1990simulated,kirkpatrick1983optimization,aarts1989simulated}. Greedy algorithm is the another type of heuristic methods that makes the locally optimal choice at each step and is applied on the traveling salesman problem \cite{gutin2006traveling}, the minimum spanning tree problem \cite{prim1957shortest} and etc.
A typical greedy method is local search, which explores the search space in the neighborhood of the current solution and selects the best neighboring solution \cite{baum1987towards,martin1991large,mladenovic1997variable,johnson1997traveling}.

In recent years, there has been a growing interest in applying machine learning methods to solve combinatorial optimization problems. The early work was based on pointer networks\cite{vinyals2015pointer}, which leverage sequence-to-sequence models to produce permutations over inputs of variable size, as is relevant for the canonical travelling salesman problem (TSP). Since then, numerous studies have fused GNNs with various heuristics and searching  procedures to solve specific combinatorial optimization problems, such as graph matching\cite{bai2019simgnn}, graph colouring\cite{lemos2019graph} and the TSP\cite{li2018combinatorial}.
Recent works aim to train neural networks in an unsupervised, end-to-end fashion, without the need for labelled training sets.  Specifically, RUN-CSP\cite{toenshoff2019run} solves optimization  problems that can be framed as maximum constraint satisfaction  problems. In \cite{yao2019experimental,karalias2020erdos}, a GNN is trained to specifically solve the MaxCut problem on relatively small graph sizes with up to 500 nodes.

\subsection{Our contributions}
In this work, we propose an efficient and stable framework, called Monte Carlo policy gradient (MCPG), to solve binary optimization problems by stochastic sampling. Our main contributions are summarized as follows.

\begin{itemize}
    \item \textbf{A novel probabilistic approach}.   We first construct a probabilistic model to sample binary solutions from a  parameterized  policy distribution. While the Gibbs distribution can locate optimal binary solutions, it suffers from high computational complexity of sampling for a complicated function. Therefore, we introduce a parameterized policy distribution with improved sampling efficiency and a stochastic optimization problem is obtained by minimizing the KL divergence between the parameterized policy distribution and the Gibbs distributions. Due to the parameterization on the distribution, an explicit formulation of the gradient  can be derived similarly to the policy gradient in reinforcement learning. Consequently, a policy gradient method with stochastic sampling is applied to find the parameters of the policy distribution and the parameterized policy distribution can guide us to search for high-quality binary solutions.
    \item \textbf{An efficient parallel sampling scheme}. Instead of direct sampling, we employ a parallel MCMC method to sample from the policy distribution and compute the policy gradient with MCMC samples. The parallel MCMC method takes advantage of both high-quality initialization and massive parallelism. The Markov property of samples maintains consistency of the sampling process and the parallelization enables a wide range of exploration for binary solutions.  A filter function is further developed  in sampling to flatten the original binary optimization problem within a neighborhood.  Based on local search methods, the filter technique improves the quality of binary solutions and  reduces the potential difficulty to deal with the functions in discrete spaces.
    \item \textbf{Theoretical guarantees for near-optimal  solutions}. 
          First, the probabilistic model with a smaller expectation of the function value is shown to have a higher probability for sampling optimal points. Then, we demonstrate that the filter function may lead to smaller function values with high probability. Using the concentration inequality of Markov chains, it is proven that the policy gradient method converges to stationary points in expectation. These results partly explain why our method can provide near-optimal solutions with efficiency and stability.

    \item \textbf{A simple yet practical implementation}. Our algorithm can outperform other state-of-the-art algorithms in terms of solution quality and efficiency on quite a few problems such as MaxCut, MIMO detection, and MaxSAT. Specifically, for the MaxCut problem, MCPG achieves better solutions than all previously reported results on two instances of the public dataset Gset, indicating that MCPG can find optimal solutions.
\end{itemize}

Additionally, our approach can also be extended to solve the general constrained binary optimization problem:
\begin{equation}
    \label{eq:constprob}
    \begin{aligned}
        \min_{ x } \quad  f(x), \;  \st  \; c(x) = 0,\; x\in \Bcal_n,
    \end{aligned}
\end{equation}
where
$c(x)=(c_1(x),\ldots,c_m(x))^\top \in \mathbb{R}^m$.

\subsection{Comparison with Other Probabilistic Methods}
Sampling the binary solution from well-designed distribution is not rare in solving binary optimization problems. For example, the Goemans-Williamson (GW) algorithm \cite{goemans1995improved} uses the randomized rounding technique in solving binary quadratic problems based on semidefinite relaxation. Since the SDP solution can be decomposed as $X=V^{T}V$ where $V=[v_1,\dots,v_n]$ with $v_i$ on the unit sphere, the $i$-th component of the binary solution is obtained by taking the sign of the inner product between $v_i$ and a random vector $r$ uniformly on the sphere. It has been proven that the GW algorithm has a $0.878$ approximate rate on the MaxCut and Max2SAT problems \cite{goemans1995improved,goemans1994879}. From then on, variants of the GW algorithm have been proposed for MIMO detection problems \cite{mobasher2007near, jiang2021tightness} and MaxCut problems \cite{krislock2014improved, rendl2010solving}.
While the sampling distribution is fixed in the GW algorithm for randomized rounding, our  parameterized distribution is refined iteratively in the policy gradient procedure.

Many learning methods for combinatorial optimization also appears with probabilistic methods. Erdos \cite{karalias2020erdos} is one of the typical learning methods, which proposes an unsupervised learning approach for combinatorial optimization problems on graphs. It utilizes the distribution over nodes parameterized by a graph neural network and performs mini-batch gradient descent to a probabilistic
penalty loss function. Many of the unsupervised learning algorithms have  similar framework as Erdos. PI-GNN \cite{schuetz2022combinatorial} applys a differentiable loss function based on the relaxation strategy to the problem Hamiltonian, where the probability over nodes is also represented by a graph neural network.
They sample directly from the distribution over nodes and generate binary solutions with an extra decoding step. Distinguished from these methods, we establish an entropy-regularized probabilistic model with a parameterized sampling policy and employ a MCMC sampling algorithm to keep the consistency among the binary solutions.

\subsection{Organization}
The rest of this paper is organized as follows. In Section \ref{sec:model}, we propose a  probabilistic model and adopt a policy gradient method. Then, a universal prototype algorithm is given. In Section \ref{sec:MCPG}, we modify our MCPG framework with a filter function. In Section \ref{sec:theory}, we investigate  theoretical results on the local search technique and the  convergence of MCPG. Numerical experiments on various problems are reported in Section \ref{sec:numerical}. Finally, we
conclude the paper in Section \ref{sec:conclusion}.

\section{Probabilistic model for binary optimization}
\label{sec:model}
In this section, we will build the probabilistic model for solving binary optimization problems. A discussion that optimal points can be found by Gibbs distributions is first presented. It motivates us to introduce an extra parameterized policy distribution as an approximation of Gibbs distributions. We construct the objective function by minimizing the KL divergence between these two distributions. The mean field (MF) approximation is used to reduce the complexity of our model. Finally, we describe a policy gradient method to find the optimal sampling policy.

\subsection{Parameterized probabilistic model}
Since binary optimization problem \eqref{prob:intro} is typically NP-hard, finding a point $x$ of high quality poses a significant challenge. The goal is to search for the set of optimal points $\xopt$ in the discrete space $\Bcal_n$, where continuous optimization algorithms cannot operate directly. A probabilistic approach is to consider an indicator distribution, i.e.,
\begin{equation*}
    q^{*}(x)=\frac{1}{|\xopt|}\mathbbm{1}_{\xopt}(x)=\begin{cases}
        \frac{1}{|\xopt|}, & \quad x\in \xopt,     \\
        0,                 & \quad x\not\in \xopt.
    \end{cases}
\end{equation*}
The distribution $q^{*}$ is the ideal distribution where only the optimal points have the uniform probability. However, the optimal points set $\xopt$ is unknown. To approach $q^{*}$, we introduce the following family of Gibbs distributions:
\begin{equation}
    \label{eq:Gibbs}
    q_\lambda(x) = \frac{1}{Z_\lambda}\exp\left(-\frac{f(x)}{\lambda}\right),\quad x\in\Bcal_n,
\end{equation}
where $Z_\lambda = \sum_{x\in\Bcal_n}\exp\left(-\frac{f(x)}{\lambda}\right)$ is the normalization factor. Given the optimal value $f^*$ of the objective function, for any $x\in \Bcal_n$, we notice that
\begin{equation}
    \label{eq:Gibbs_approx}
    \begin{aligned}
        q_{\lambda}(x)
         & = \frac{\exp\left(\frac{f^*-f(x)}{\lambda}\right)}{\sum_{x\in\Bcal_n}\exp\left(\frac{f^*-f(x)}{\lambda}\right)} =\frac{\exp\left(\frac{f^*-f(x)}{\lambda}\right)}{|\xopt|+\sum_{x \in \Bcal_n/\xopt}\exp\left(\frac{f^*-f(x)}{\lambda}\right)} \\
         & \to \frac{1}{|\xopt|}\mathbbm{1}_{\xopt}(x),\quad  \text{as } \lambda\to 0.                                                                                                                                                                    \\
    \end{aligned}
\end{equation}
Although $q^{*}$ is not computable in practice, the Gibbs distribution $q_{\lambda}$, which converges to $q^{*}$ pointwisely, provides a feasible alternative for searching optimal solutions.

Motivated by this observation, one can devise an algorithm to randomly sample from the Gibbs distribution and control the parameter $\lambda$ to approximate the global optimum. However, since $q_{\lambda}(x)$ needs to compute the summation of $2^{n}$ terms in the denominator $Z_{\lambda}$, the complexity of direct sampling grows exponentially. Meanwhile, $q_{\lambda}$ is hard to sample when $f$ owns a rough landscape. Instead of direct sampling from the Gibbs distribution, we approximate it using a parameterized policy that can be sampled efficiently. The parameterized distribution $p_{\theta}(x|\mathcal{P})$ is designed to have a similar but regular landscape to that of $q_{\lambda}$ to enable efficient sampling. For notational convenience, we write $p_{\theta}(x)=p_{\theta}(x|\mathcal{P})$.

To measure the distance between $p_{\theta}$ and $q_{\lambda}$, we introduce the KL divergence  defined as follows:
\begin{equation*}
    \KL{p_\theta}{q_\lambda}= \sum_{x\in\B_n}p_\theta(x)\log\frac{p_\theta(x)}{q_\lambda(x)}.
\end{equation*}
In order to reduce the discrepancy between the policy distribution $p_\theta$ and the Gibbs distribution $q_\lambda$, we minimize this KL divergence. We notice that
\begin{align}
    \label{eq:KL}
    \min_{\theta}\quad\KL{p_\theta}{q_\lambda} & = \frac{1}{\lambda}\sum_{x\in\B_n}p_\theta(x) f(x)+\sum_{x\in\B_n}p_\theta(x)\log p_\theta(x)+\log Z_\lambda\nonumber \\
                                               & =\frac{1}{\lambda} \prn{\ex{p_\theta}{ f(x)}+\lambda\ex{p_\theta}{\log p_\theta(x)}}+\log Z_\lambda.
\end{align}
Since $Z_\lambda$ is a constant, \eqref{eq:KL} is equivalent to the following problem:
\begin{equation}
    \label{prob:loss}
    \min_{\theta}\quad L_\lambda(\theta) = \underbrace{\ex{p_\theta}{f(x)}}_{L(\theta)}+\lambda\underbrace{\ex{p_\theta}{\log p_\theta(x)}}_{H(\theta)}.
\end{equation}
While minimizing \eqref{prob:loss} is to find an approximation of Gibbs distribution $q_\lambda$, one can explain \eqref{prob:loss} from the perspective of optimization. In \eqref{prob:loss}, the first term is the expected loss of $f(x)$ with $x\sim p_\theta$ and the second term $\ex{p_\theta}{\log p_\theta(x)}$ is the entropy regularization of $p_{\theta}$. Minimizing the first term $L(\theta)$ yields the optimal value $f^*$ when $p_{\theta}(x\in \xopt)=1$. However, it leads to hardship on sampling as $p_{\theta}$ is rough. By adding the entropy regularization, we encourage exploration and promote the diversity of samples. On the other hand, the regularization parameter $\lambda$, which is also called annealing temperature, progressively decreases from an initial positive value to zero as in the classical simulated annealing, and the global optimum are finally obtained after  a sufficiently long time.

\subsection{Parameterization of sampling policy}
As we discuss above, we can use probabilistic methods to solve binary optimization problems. The sampling policy can be parameterized by a neural network  with $\mathcal{P}$ containing features of the problem, that is, \begin{equation}
    \label{eq:neuralpolicy}
    p_{\theta}(x|\mathcal{P})\propto e^{\phi_{\theta}(x,\mathcal{P})},
\end{equation}
where $\phi$ represents a neural network with input $x$ and problem instance $\mathcal{P}$. It is unrealistic to generate a normalized distribution over $2^n$ points in $\Bcal_{n}$. The unnormalized policy \eqref{eq:neuralpolicy} can be sampled  through the Metropolis algorithm. The neural network can be designed to reduce the computational complexity in the sampling process.

One natural idea is to consider the distribution on each single variable and assume independence, which is called the mean field (MF) approximation in statistical physics. The MF methods provide tractable approximations for the high dimensional computation in probabilistic models. By neglecting certain dependencies between random variables, we simplify the parameterized distribution $p_{\theta}$ with the independent assumption for each component of $x$. As for binary optimization problems $\mathcal{P}$, the parameterized probability follows the multivariate Bernoulli distribution:
\begin{equation}
    \label{eq:mf}
    p_{\theta}(x)=\prod_{i=1}^{n}\mu_i^{(1+x_i)/2}(1-\mu_i)^{(1-x_i)/2}	,~\mu_i=\frac{1}{1+e^{-\theta_i}},
\end{equation}
where $\theta_i, ~i=1,\dots,n$ are parameters.
Under the mean field approximation, $x_1,x_2,\dots,x_n$ are viewed independent and $\mu_i$ is the probability that the $i$-th variable takes value $1$. 
The MF approximation provides an efficient way to encode binary optimization problems and greatly reduces the complexity of the model from $2^n$ to $n$. Besides, the distribution under MF approximation is beneficial for the Markov chain Monte Carlo sampling.

\subsection{Prototype Algorithm}
\label{sec:primitive}
In the above discussion, we have already constructed the loss function $L_{\lambda}(\theta)$ and a mean field sampling policy $p_{\theta}(x)$. To train the policy, the gradient formulation is derived explicitly. Then, we purpose the brief algorithm framework. The sampled solutions help to compute the policy gradient and improve the sampling policy. The policy is optimized by the stochastic policy gradient, which guides us to sample those points with lower function values. Thus, the binary optimization problems are solved by stochastic sampling based on a neural network policy.

\begin{lemma}\label{lem:loss_grad}
    Suppose for any $x\in\Bcal_n$, $p_\theta(x)$ is differentiable with respect to $\theta$. For any constant $c\in\real$, the gradient of the loss function \eqref{prob:loss} is given by
    \begin{equation}
        \label{prob:loss_gradient}
        \nabla_\theta L_\lambda(\theta) = \ex{p_\theta}{(f(x)+\lambda\log p_\theta(x)-c)\nabla_\theta\log p_\theta(x)}.
    \end{equation}

\end{lemma}
\begin{proof}
    By exchanging the order of differentiation and summation, we note that
    \begin{equation}
        \label{eq:int_0}
        \begin{aligned}
            \ex{p_\theta}{\nabla_\theta\log p_\theta(x)}
            = \sum_{x\in\Bcal_n}\nabla_\theta p_\theta(x)=\nabla_\theta\sum_{x\in\Bcal_n}p_\theta(x)=\nabla_\theta 1=0,
        \end{aligned}
    \end{equation}
    as $\sum_{x\in\Bcal_n}p_\theta(x)=1$ is a constant. The gradient of the loss function \eqref{prob:loss} is derived by
    \begin{align*}
        \nabla_\theta L_\lambda(\theta) 
         & = \nabla_\theta\sum_{x\in\Bcal_n}p_\theta(x)\prn{ f(x)+\lambda\log p_\theta(x)}                                                                                \\
         & = \lambda\sum_{x\in\Bcal_n}p_\theta(x)\nabla_\theta\log p_\theta(x)+\sum_{x\in\Bcal_n}\prn{ f(x)+\lambda\log p_\theta(x)}\nabla p_\theta(x)                    \\
         & \overset{(i)}{=}\lambda\sum_{x\in\Bcal_n}\nabla_\theta p_\theta(x)+\sum_{x\in\Bcal}\prn{ f(x)+\lambda\log p_\theta(x)}p_\theta(x)\nabla_\theta\log p_\theta(x) \\
         & \overset{(ii)}{=}\sum_{x\in\Bcal}\prn{ f(x)+\lambda\log p_\theta(x)}p_\theta(x)\nabla_\theta\log p_\theta(x)                                                   \\
         & =\ex{p_\theta}{\prn{ f(x)+\lambda\log p_\theta(x)}\nabla_\theta \log p_\theta(x)},
    \end{align*}
    where $(i)$ uses the substitution $\nabla_{\theta} p_{\theta}(x)=p_{\theta}(x)\nabla_{\theta}\log p_{\theta}(x)$ and $(ii)$ is due to \eqref{eq:int_0}. For any given $c$ irrelevant to $x$, \eqref{eq:int_0} yields that
    $\ex{p_\theta}{c\nabla_\theta\log p_\theta(x)}=0$. This completes the proof.
\end{proof}

Especially, we take $c=\ex{p_\theta}{f(x)}$ and denote $A_{\lambda}(x;\theta)=f(x)+\lambda\log p_\theta(x)-\ex{p_\theta}{f(x)}$. We call the function $A_\lambda(x;\theta)$ as the ``advantage function'' analogous to the one in reinforcement learning \cite{williams1992simple}.  It has been shown that such technique can reduce the variance in our training and make it easier to sample points with lower cost. Hence, the stochastic gradient is computed by
\[
    \begin{aligned}
        A_\lambda(x;\theta,S) & =f(x)+\lambda\log p_\theta(x)-\sum_{x\in S}f(x),                      \\
        \bar{g}(\theta,S)     & =\sum_{x\in S}A_\lambda(x;\theta,S)\nabla_{\theta}\log p_{\theta}(x),
    \end{aligned}
\]
where $S$ is a sample set from the distribution $p_{\theta}(x)$.

Using the probabilistic model, we propose a prototype algorithm for solving binary optimization problems as follows:
\begin{enumerate}
    \item At the beginning of each iteration, the probabilistic model provides a distribution $p_\theta(\cdot | \mathcal{P})$, indicating where the good solution potentially lies.
    \item A set of points $S$ are sampled from $p_\theta(\cdot | \mathcal{P})$.
    \item The policy gradient is computed in order to update the probabilistic model, providing an improved distribution for the next iteration.
    \item The parameter $\theta$ in the probabilistic model is updated as follows
          \begin{equation*}
              \theta^{t+1} = \theta^t - \eta^t \bar{g}(\theta^t, S^t),
          \end{equation*}
          where $\eta^t$ and $S^t$ are the step size and the sample set at the $t$-th iteration.
\end{enumerate}

Algorithm \ref{alg:Proto} gives the pseudo-code of the prototype algorithm. The prototype algorithm is efficient in obtaining better solutions compared to most naive heuristic methods, since similar  frameworks are commonly used in machine learning. However, it also faces several critical challenges. It is difficult to determine whether a good solution has been achieved.
Moreover, keeping the diversity of samples remains challenges at the later iterations. On the other hand, due to the NP-hard nature of the problem \eqref{prob:intro}, there are many local minimum in solving problem \eqref{prob:loss}, making it challenging for policy gradient methods to find global minima. This difficulty is prevalent in most reinforcement learning scenarios.

\begin{algorithm}[htbp]
    \label{alg:Proto}
    \caption{Prototype algorithm based on the probabilistic model}
    \label{alg:MCPG-simple}
    \SetAlgoNoEnd
    \KwIn{problem instance $\mathcal{P}$, probabilistic model $p_\theta(\cdot|\mathcal{P})$, number of epochs $\tau$}
    \For{$t=1$ to $\tau$}{
        Obtain $p_\theta(\cdot|\mathcal{P})$ from the probabilistic model\;
        Obtain set $S^t$ sampling from $p_\theta(\cdot|\mathcal{P})$\;
        Compute the advantage and the gradient using $S^t$\;
        Update $\theta$ using policy gradient on $S^t$\;
    }
    \textbf{return}	the best found solution in samples\;
\end{algorithm}

\section{A Monte Carlo policy gradient method}
\label{sec:MCPG}
The probabilistic model successfully transforms binary optimization problems into training the policy for sampling. Since neural networks help us sample the lower function value from discrete spaces, there are technical details to make the problem-solving process more efficient. In this section, we introduce our MCPG  method for solving binary optimization problems \eqref{prob:intro} using the probabilistic model \eqref{prob:loss}. It  consists of the following main components:
\begin{enumerate}
    \item a filter function $T(x)$ that enhances the objective function, reducing the probability of the algorithm from falling into local minima;
    \item a sampling procedure with filter function, which starts from the best solution found in previous steps and tries to keep diversity;
    \item a modified policy gradient algorithm to update the probabilistic model;
    \item a probabilistic model that outputs a distribution $p_\theta(\cdot|\mathcal{P})$, guiding the sampling procedure towards potentially good solutions.
\end{enumerate}

The pipeline of MCPG is demonstrated in Fig. \ref{pipeline}. In each iteration, MCPG starts from the best samples of the previous iteration and performs MCMC sampling in parallel. The algorithm strives to obtain the best solutions with the aid of the powerful probabilistic model. To improve the efficiency, a filter function is applied to compute a modified objective function. At the end of the iteration, the probabilistic model is updated using policy gradient, ensuring to push the boundaries of what is possible in the quest for optimal solutions.

\begin{figure}[htbp]
    \centering
    \includegraphics[width = \linewidth]{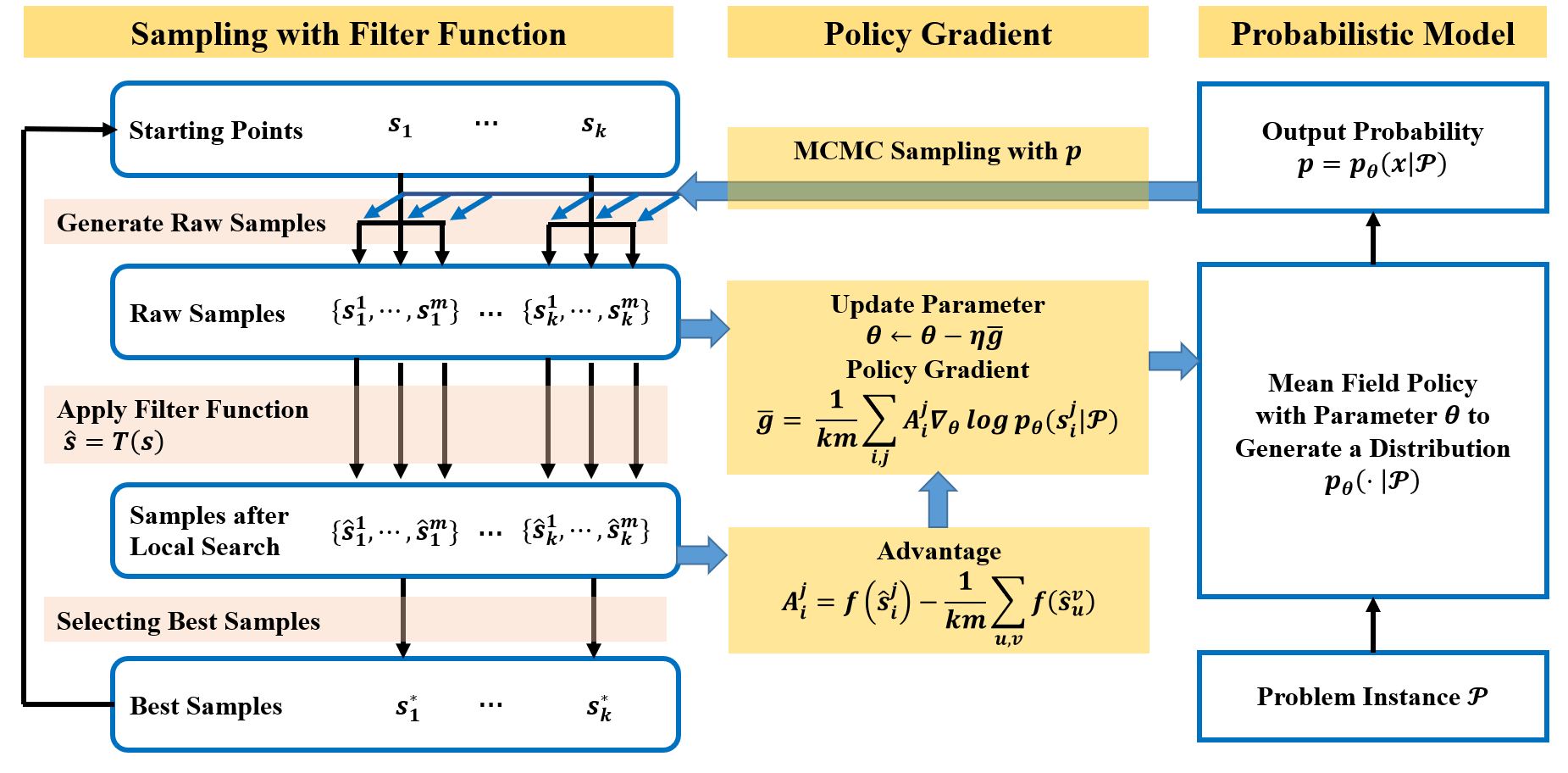}
    \caption{The pipeline of MCPG.}
    \label{pipeline}
\end{figure}

In the next few subsections, we first discuss the filter function. Then we present the sampling method and the update scheme in detail.

\subsection{The Filter Function}\label{sec:local_search}
The filter function $T(x)$ is a critical component of MCPG, which in essense maps an input $x$ to a better point $\hat{x}$. Since the specific details of the function $T$ depend on the structure of the problem, we next give a general definition.

\begin{definition} [Filter Function]
    \label{filterdefinition}
    For each $x \in \mathcal B_n$, let $\mathcal N(x)\subset \mathcal B_n$ be a neighborhood of $x$  such that $x\in \mathcal N(x)$,  $|\mathcal N(x)|\geq 2$ and  any point in $\mathcal N(x)$ can be reached by applying a series of ``simple" operations to $x$.  Then we define a filter function $T(x)$ induced by the neighborhood  $\mathcal N(x)$  as
    \begin{equation*}
        T(x)\in\mathop{\arg\min}\limits_{\hat{x}\in \mathcal N(x)}f(\hat{x}),
    \end{equation*}
    where $T(x)$ is arbitrarily chosen if there exists multiple solutions.
\end{definition}
With the filter function $T$, the original binary optimization problem is converted to
\begin{equation}
    \label{prob:prob_filter}
    \begin{array}{cc}
        \min & f(T(x)),    \quad
        \st  \quad               x\in \Bcal_n.
    \end{array}
\end{equation}
For an optimal point $x^*$ of \eqref{prob:intro}, it holds that $x^{*}= T(x^{*})$ and $f(x^*) = f(T(x^*))$. We further have $f(T(x)) \ge f(x^*)$ because $T(x) \in  \Bcal_n$, which implies that $x^*$ is also the optimal point of $f(T(x))$. Therefore, the problem \eqref{prob:prob_filter} has the same minimum value as \eqref{prob:intro}. The filter function maintains the global optimal solution of $f(x)$. At the same time, it smoothens $f(x)$ and mitigates the common problem that the probabilistic model falls into a local minima during the iteration. We will demonstrate the effect of filter function at the end of this subsection.

There are many ways to choose the neighborhood $\mathcal N(x)$ and the filter $T$. For example, $T$ can be the projection from $x$ to the best point among its neighbors where at most $k$ variables differ from $x$, i.e.,
\begin{equation}
    \label{T1}
    T_k(x) \in \mathop{\arg\min}\limits_{\|\hat{x}-x\|_{1} \le 2k} f(\hat{x}). 
\end{equation}
The above procedure can be performed efficiently in parallel for all samples. Moreover,
an algorithm designed to improve an input point $x$ can also serve as a filter function, even if the neighborhood set $\mathcal N(x)$ cannot be formulated explicitly. For example,  we can introduce a greedy algorithm called Local Search (LS):
\begin{equation}
    \label{T_Greedy}
    T_{LS}(x) = \mathtt{LocalSearch}_f(x).
\end{equation}
LS starts with a feasible point $x^{(0)}$ and then repeatedly explores the nearby points by making small modifications to the current point. It will move to a better nearby point until a locally optimal solution is found. The pseudo-code of LS is demonstrated in Algorithm \ref{alg:localsearch}. Although the neighborhood $\mathcal{N}(x)$ can not be explicitly expressed in LS, it constructs a search tree during iterations. The search tree covers a wide range of points that is nearby to the input $x$.

\begin{algorithm}[htbp]
    \label{alg:localsearch}
    \caption{Local Search}
    \SetAlgoNoEnd
    \KwIn{Feasible point $x$, permutation of the variable indices $I = \{i_1,i_2,\cdots,i_n\}$}
    \For{$t=1$ to $n$}{
        Choose a single variable from the current point $x$ based on permutation $I$\;
        Consider the nearby set of $x$ on the index $i_t$
        \begin{equation*}
            \Ncal(x,i_t) = \left\{x^{\prime} ~|~ x^{\prime}_{i_t} \in \{-1,1\}, x^{\prime}_{j} = x_{j},\forall j \in \{1,2,\cdots,n\}\backslash\{i_t\}\right\};
        \end{equation*}\\
        Determine if the flipping operation is an improvement. We
        find the best point within $\Ncal(x^{(t-1)}, i_t)$, i.e.,
        \begin{equation*}
            x^{(t)} = \mathop{\arg\min}\limits_{x\in \Ncal(x^{(t-1)}, i_t)} f(x);
        \end{equation*}
    }
    \textbf{return}	$T_{LS}(x)=x^{(n)}$\;
\end{algorithm}

We demonstrate the effect of the filter functions $T_{k}$ and $T_{LS}$ on the MaxCut problem. The MaxCut problem on the graph $G = (V,E)$ can be expressed as the following format:
\begin{equation*}
    \max  \quad  \sum_{(i,j) \in E} w_{ij} (1-x_i x_j),
    \quad \st   \quad  x\in \{-1, 1\}^n.
\end{equation*}
The effect of the filter functions $T_k(\cdot)$ and $T_{LS}(\cdot)$ is shown in Fig. \ref{fig:T}. The experiments are carried out on the Gset instance G22, which has 2000 nodes and 19990 edges. We first show the expectation of the objective function before and after applying the filter function. The expectation is approximated by uniform sampling. As illustrated in Section \ref{sec:result_prob}, the probabilistic model gains better performance as the expectation gets closer to the optimum. From Fig. \ref{fig:T}(a), we find out that the expectation grows rapidly with such a simple filter function.

To illustrate the influence of the filter function, we select 20 solutions $\{x_1, \cdots, x_{20}\}$ and depict the change of the objective function in Fig. \ref{fig:T}(b). Each solution has only one variable differing from the preceding one, i.e., $\|x_{i+1}-x_{i}\|_1 = 2$. These points construct a “cross section" of the MaxCut problem. After applying the filter function, the function becomes smoother, leading to a reduction in the number of local maxima.

\begin{figure}[htbp]
    \centering  
    \subfigure[The expectation of the objective function.]{
        \includegraphics[scale=0.42]{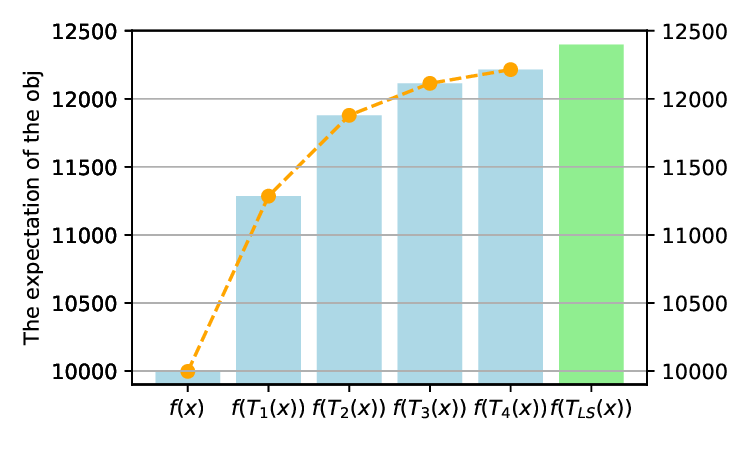}
    }
    \subfigure[A selected sequence of solutions.]{
        \includegraphics[scale=0.42]{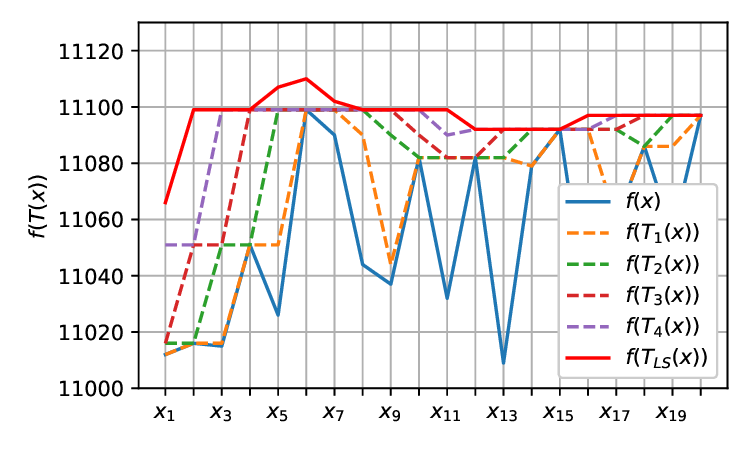}
    }
    \caption{The effect of the filter function for MaxCut problem on the Gset instance G22.}
    \label{fig:T}
\end{figure}

\subsection{Parallel sampling methods with filter function}\label{sec:mcmc}

In this section we introduce our sampling method. To facilitate the application of the policy gradient, a substantial number of samples are required. Although binary solutions are relatively straightforward to obtain, direct sampling based on the Bernoulli distribution at each iteration precludes the use of previously obtained samples. Nevertheless, ``old" samples are valuable in aiding the algorithmic process.

To utilize old samples, we combine the filter function with MCMC sampling techniques to enhance the quality of our samples. Algorithm \ref{HM-single} shows our sampling procedure for a single starting point, which is based on the Metropolis-Hastings (MH) algorithm to establish a Markov chain with the stationary distribution $p_\theta(\cdot)$ and is integrated with the filter function.

\begin{algorithm}[htbp]
    \small
    \SetAlgoLined
    \SetAlgoNoEnd
    \caption{A Parallel MH algorithm with filter function}
    \label{HM-single}
    \KwIn{The starting state $x_0$, transition number $t$, number of chain $m$, proposal distribution $Q(x^{\prime}|x)$.}
    \SetKwFunction{FunctionName}{sampling}
    \SetKwProg{Fn}{Function}{:}{}
    \SetKwProg{ParFor}{for}{ do in parallel}{end}
    \Fn{\FunctionName{$x_0$, $t$, $m$}}{
    \ParFor{$j = 1$ \KwTo $m$}{
    Copy the starting state $x_0^j = x_0$ for this chain\;
    \ParFor{$v = 0$ \KwTo $t-1$}{
    Propose a new state $x^{\prime}$ by sampling from a proposal distribution $Q(x^{\prime} | x_v^{j})$\;
    Compute the acceptance probability $\alpha(x^{\prime} | x_v^{j}) = \min\left(1, \frac{p_\theta(x_v^{j})Q(x_v^{j} | x^{\prime})}{p_\theta(x^{\prime})Q(x^{\prime} | x_v^{j})}\right)$\;
    Generate a random number $u \sim \text{Uniform}(0,1)$\;
    \eIf{$u < \alpha(x^{\prime} | x_v^{j})$}{Set $x_{v+1}^{j} = x^{\prime}$\;
    }{Set $x_{v+1}^{j} = x_v^{j}$\;}

    }
    Obtain $s^j = x_{t}^{j}$\;
    Apply filter function and obtain $\hat{s}^j = T(s^j)$\;
    }
    \textbf{return} the sample set $S=\{s^1, s^2, \cdots, s^m\}$ and $\hat{S}=\{\hat{s}^1, \hat{s}^2, \cdots, \hat{s}^m\}$\;
    }
\end{algorithm}

In MCPG, we choose $k$ starting points, which is selected from the best samples in the last iteration or randomly generated for the first iteration, and construct $m$ Markov chains for each of them. The computation of the $km$ chains can be carried out in parallel on the modern computational device, i.e., the outer loop is actually executed in parallel in Algorithm \ref{HM-single}.

With the distribution $p_{\theta^t}(\cdot)$, the transition procedure is the same as the ordinary MH algorithm. For all the chains, we only select the last states of a chain, and obtain a sample batch $	S_i=\{s_i^1,s_i^2,...,s_i^m\}$ for each of the starting point $s_i$.
Hence, we get $k$ batches and obtain $km$ raw samples $s_i^j$ for $1\leq i\leq k, 1\leq j\leq m$.

As shown in Fig. \ref{pipeline}, the filter function is used to enhance the objective function. To fit this modification, we apply the filter function on $s_i^j$ to obtain $\hat{s}_i^j = T(s_i^j)$, and these points are denoted as $\hat{S}$. This procedure can be performed efficiently in parallel for each Markov chain. Next, we compute $f(\hat s_i^j)$ and the advantage function for the policy gradient, which will be introduced in Section \ref{sec:algorithm_framework}.

In contrast to the standard Metropolis-Hastings algorithm, the sampling procedure in MCPG begins at multiple initial points and only retains the final states from a set of short Markov chains. The reason for this is that, parallel execution of the chains increases the diversity of the generated samples. At the same time, it saves computational time when compared to collecting states from a single but  long Markov chain.

The samples in $\hat{S}$ serve an additional purpose beyond their use in the policy gradient. In practice, the samples in $\hat{S}$ are informative in identifying areas where high-quality points are likely to be found. As a result, these points are selected as initial starting points in the next iteration. To ensure the efficacy of these starting points, the maximum transition time for a single chain, represented by $t$, is set to a relatively small value. The reason is that, the samples that pass through the filter function are typically of high quality, and the adoption of a small $t$ value strikes a balance between  simulating the distribution $p_\theta(x)$ and preserving similarity between starting points and samples.

\subsection{Algorithm framework}\label{sec:algorithm_framework}

Given a problem instance $\mathcal{P}$, our algorithmic idea is to sample the solution based on the distribution $p_{\theta^t}(\cdot)$ given by the probabilistic model, and use the samples to improve the probabilistic model in turn. With the filter function, MCPG focuses on the  modified binary integer programming problem \eqref{prob:prob_filter}.

In fact, we can treat $f(T(\cdot))$ as a whole and denote it by $\hat{f}(\cdot)$. When replacing $f$ by $\hat{f}$, the derivation in Section \ref{sec:model} is still valid. Following this observation, we illustrate in detail how to update the parameters in the probabilistic model.
The probabilistic model is rewritten as:
\begin{equation}
    \label{prob:loss_filter}
    \min_{\theta}\quad L_\lambda(\theta) = {\ex{p_\theta}{f(T(x))}}+\lambda{\ex{p_\theta}{\log p_\theta(x)}}.
\end{equation}
According to the policy gradient derived in Lemma \ref{lem:loss_grad} in Section \ref{sec:primitive}, we have
\begin{equation}
    \label{eq:gradient}
    \nabla_{\theta} L_\lambda(\theta)=\ex{p_{\theta}}{A_{\lambda}(x;\theta)\nabla_{\theta}\log p_{\theta}(x)},
\end{equation}
where $A(x;\theta)=f(T(x))+\lambda\log p_\theta(x)-\ex{p_{\theta}}{f(T(x))}$.

We next present a detailed description of the MCPG algorithm framework.
MCPG starts from a set of randomly initialized points. For each iteration, a sample set $S=\{s_i^j|1\leq i\leq k,1\leq j\leq m\}$ is obtained from the sampling procedure introduced in Section \ref{sec:mcmc}. The filter function $T_{LS}$ is then applied to each of the solutions in the sample set $S$ to obtain $\hat{S}$. 
We first compute the advantage function
\begin{equation}
    \label{alg:adv}
    A_\lambda(s_i^j;\theta^t, S) = f(\hat s_i^j) + \lambda \log p_{\theta^t}( s_i^j) - \frac{1}{km}\sum_{u=1}^k\sum_{v=1}^m f(\hat s_u^v),
\end{equation}
and then the policy gradient on sample set $S$ is given by
\begin{equation}
    \label{alg:gradient}
    \bar g_\lambda(\theta^t,S) = \frac{1}{km}\sum_{i,j}A_\lambda(s_i^j;\theta^t, S)\nabla_{\theta}\log p_{\theta^t}(s_i^j).
\end{equation}
At the $t$th iteration,  the parameter $\theta^t$ is then updated after the sampling procedure and obtaining the policy gradient by
\begin{equation}
    \label{eq:update2}
    \theta^{t+1} = \theta^t - \eta^t\bar g_\lambda(\theta^t, S).
\end{equation}
We can also use other stochastic gradient methods to update the parameters. Our MCPG framework is outlined in  Algorithm \ref{alg:MCPG}. Note that the two inner loops in Algorithm \ref{alg:MCPG} can be executed in parallel.

\begin{algorithm}
    \caption{MCPG}
    \label{alg:MCPG}
    \SetAlgoNoEnd
    \SetKwProg{ParFor}{for}{ do in parallel}{end}
    \KwIn{number of starting points $k$, filter function $T$, number of epochs $\tau$.}
    Initialize the probabilistic model with parameters $\theta$\;
    Initialize points $s_1,...,s_k$\;
    \For{epoch $=1$ \KwTo $\tau$}{
        Obtain distribution $p_{\theta^t}(\cdot)$ from probabilistic model.\;
        \ParFor{$i = 1$ \KwTo $k$}{
            $S_i,~ \hat{S}_i = \mathtt{sampling}(s_i, t, m)$\;
        }
        Compute advantage function $A_{\lambda}(\cdot)$ on all samples in $\hat{S}_1,\cdots,\hat{S}_k$ via \eqref{alg:adv}\;
        Update $\theta$ using $S_1,\cdots,S_k$ and $A_{\lambda}(\cdot)$ via \eqref{eq:update2} and \eqref{alg:gradient}\;
        Set starting point $s_i$ to be the best of samples in $\hat{S}_i$ for $i=1,\ldots,k$.\;
    }
    \textbf{return} the best solution.
\end{algorithm}

The influence of the filter function deserves to be emphasized. Maintaining sample diversity is crucial when using policy gradients. Without applying the filter function, $p_\theta(\cdot)$ quickly converges to local optima corresponding a few binary points, making subsequent sampling very challenging. Traditional reinforcement learning methods introduce highly technical regularization methods to avoid this problem. In contrast, the filter function in MCPG deterministically maps a large number of samples to the same local optima, thereby ensuring that $p_\theta(\cdot)$ does not quickly converge to a local minima. This approach guarantees diversity in subsequent sampling and facilitates continuous improvement of the model.

In the context of our well-designed sampling procedure, updating the probabilistic model involves some differences from the ordinary case. There exists bias between the distribution $p_\theta(\cdot)$ and the empirical distribution that are involved in the policy gradient. This bias come from the idea that the sampling procedure is designed to tend to obtain the samples of high quality, i.e., with the better objective function values, to enhance the policy gradient method. This idea is not rare in reinforcement learning. For example, assigning different priorities to samples in the replay buffer is a commonly used technique to improve performance of the policy gradient methods. In MCPG, we start the sampling procedure from the high-quality points obtained in the previous iteration to avoid worthless samples as much as possible.

\subsection{Extension to problems with constraints}
We now consider the constrained binary optimization problem \eqref{eq:constprob} by penalizing constraint violations  to force the minimum of the penalty function to lie in the feasible region of \eqref{eq:constprob}. The $\ell_1$ penalty function is introduced as
\begin{equation}
    \label{eq:penaltyfun}
    f_{\sigma}(x):=f(x)+\sigma \|c(x)\|_1,
\end{equation}
where $\sigma>0$ is the penalty parameter. Then, our penalty problem is:
\begin{equation}
    \label{eq:penaltyprob}
    x^*_\sigma = \arg \min_{x\in \Bcal_n} f_{\sigma}(x).
\end{equation}
Consequently, the unconstrained problem \eqref{eq:penaltyprob} is  solved by MCPG stated in Algorithm \ref{alg:MCPG}.  The optimal solution of \eqref{eq:penaltyprob} will be feasible to \eqref{eq:constprob} if  a sufficiently large penalty parameter $\sigma$ is taken. In fact, due to the limited variable space, the penalty framework  \eqref{eq:penaltyprob} is exact and it has the same minimum as  \eqref{eq:penaltyprob}.
\begin{proposition}
    There exists a sufficiently large value $\bar{\sigma}$ such that for all $\sigma> \bar{\sigma}$, problems \eqref{eq:constprob} and \eqref{eq:penaltyprob} have the same global minimum.
\end{proposition}
\begin{proof}
    Suppose that $x^{*}$ and $x^*_\sigma$ are  global minimum of problems \eqref{eq:constprob} and \eqref{eq:penaltyprob}, respectively. Let $\varpi:=\min_{x\in \Bcal_n} \{ \|c(x)\|_1 \mid \|c(x)\|_1 \neq 0\}$. If there is no $x\in \Bcal_n$  such that $\|c(x)\|_1 \neq 0$, then all constraints $c_i(x)=0$ can be safely removed and \eqref{eq:constprob} is reduced to the standard problem  \eqref{prob:intro}. Hence, we can assume that $\varpi>0$, i.e.,  the smallest value of $\abs{c_i(x)}$ among all infeasible point $x$ is positive.  Let $\bar f^{*} =\min_{x \in \Bcal_n} f(x)$.  Since  $c(x^*)=0$, we have $f(x^*)=f_\sigma(x^*)$. Hence, we  can define the positive threshold value:
    \begin{equation*}
        \bar{\sigma}=(f(x^*)-\bar f^{*})/\varpi \ge 0.
    \end{equation*}

    We first show $x^{*}$ is a global minima of \eqref{eq:penaltyprob}. For any feasible point $x$ of \eqref{eq:constprob}, we have
    \begin{equation}
        \label{eq:exact1}
        \begin{aligned}
            f_{\sigma}(x)=f(x)\ge f(x^{*})=f_{\sigma}(x^{*}). 
        \end{aligned}
    \end{equation}
    On the other hand,  for a point $x \in \Bcal_n$ but $\|c(x)\|_1\neq 0$,   it holds for $\sigma > \bar{\sigma}$ that
    \begin{equation}
        \label{eq:exact2}
        \begin{aligned}
            f_{\sigma}(x)=f(x)+\sigma\|c(x)\|_1 \geq \bar f^{*}+\sigma \varpi > \bar f^{*}+(f_\sigma(x^*)-\bar f^{*})=f_\sigma(x^*),
        \end{aligned}
    \end{equation}
    where the inequalities are due to the definitions of $\bar f^{*}$ and $\varpi$. Therefore, $x^{*}$ is a global optimal solution of \eqref{eq:penaltyprob}.

    Then, we can also prove that $x^{*}_{\sigma}$ is a global minima of \eqref{eq:constprob} for any $\sigma>\bar\sigma$. If $\norm{c(x^{*}_{\sigma})}_1\not=0$, \eqref{eq:exact2} implies $f_{\sigma}(x^{*}_{\sigma})> f_{\sigma}(x^{*})$, which conflicts with the definition of $x^{*}_{\sigma}$. It shows that $x^{*}_{\sigma}$ is feasible for \eqref{eq:constprob}.  Then for any feasible $x\in \Bcal_n$, we have $f(x^{*}_{\sigma})=f_{\sigma}(x^{*}_{\sigma})\le f_{\sigma}(x)=f(x)$ by the definition of $x^{*}_{\sigma}$.  Therefore, $x^{*}_{\sigma}$ is a global minima of \eqref{eq:constprob}.
\end{proof}

\section{Theoretical analysis}
\label{sec:theory}
Although MCPG has good practical performance, we are still concerned about its theoretical properties.
In this section, we establish the relation between our proposed probabilistic solution and the optimal solution $\xopt$ in \eqref{prob:intro}. It shows that minimize the expectation of $f(x)$ will increase the probability of sampling optimal solutions. Hence, the filter function is proven to improve the probabilistic model by the  local search. Moreover, we illustrate the policy will be stationary and give convergence analysis of Algorithm \ref{alg:MCPG} from the perspective of optimization.

\subsection{Relationship between binary optimization and probabilistic model}\label{sec:result_prob}
In order to link the discrete and continuous model, we begin with the following definition to describe the properties of $f$.

\begin{definition}
    For an arbitrary function $f$ on $\Bcal_n$, we define the optimal gap $\bgap$ as
    \begin{equation}
        \label{eq:fgap}
        \bgap(f) = \min_{x\in\Bcal_n\backslash\xopt}f(x) - f^*.
    \end{equation}
    The bound of a function $f$ is denoted by
    \begin{equation}
        \label{eq:fbound}
        B(f) = \max_{x\in\Bcal_n}|f(x)| .
    \end{equation}
\end{definition}

A direct property of the definition is given as follows. It shows that when the probabilistic model is minimized enough, the obtained probability  from the optimal solutions is linearly dependent on the gap between the expectation and the minimum of $f$.

\begin{proposition}
    \label{corro:distribution_optim}
    For any $0<\delta<1$, suppose $L_{\lambda}(\theta)-f^*<(1-\delta)\bgap(f)$, then
    \[
        \mathbb{P}(x\in\xopt)>\delta.
    \]
    Therefore, for $x^1,\dots,x^m$ independently sampled from $p_{\theta}$, $\min_k f(x^k) = f^*$ with probability at least $1-(1-\delta)^m$.
\end{proposition}
\begin{proof}
    Note that
    \begin{equation*}
        \begin{aligned}
            L_\lambda(\theta) = & \ex{ p_{\theta}(x)}{f(x)}                                                                      \\
            =                   & \sum_{x\in\Bcal_n\backslash\xopt}f(x) p_{\theta}(x)+\sum_{x\in\xopt}f(x) p_{\theta}(x)         \\
            \ge                 & \sum_{x\in\Bcal_n\backslash\xopt}(f^*+\bgap(f))p_{\theta}(x)+\sum_{x\in\xopt} p_{\theta}(x)f^* \\
            =                   & f^* + (1-\Pbb(x\in \xopt))\bgap(f).
        \end{aligned}
    \end{equation*}
    Combining the above inequalities and the condition gives the conclusion.
\end{proof}

\subsection{Properties of the Filter Function}

We now  analyze the effect of the filter function $T$. Similar to  the Gibbs distributions $q_\lambda(x)$ defined by \eqref{eq:Gibbs},   we also consider the Gibbs distributions $\hat q_\lambda(x)$  for  filter functions as
\begin{equation*}
    \hat q_\lambda(x) = \frac{1}{\hat Z_\lambda}\exp\left(-\frac{f(T(x))}{\lambda}\right),\quad x\in\Bcal_n,
\end{equation*}
where $\hat Z_\lambda = \sum_{x\in\Bcal_n}\exp\left(-\frac{f(T(x))}{\lambda}\right)$ is the normalization.
In order to compare the parameterized distribution $q_\lambda$ and the modified distribution $\hat q_\lambda$, we utilize the KL divergence between these distributions and the parameterized distribution $p_\theta$. Specifically, we consider the KL divergence defined in Equation \eqref{eq:KL}.

When $T(x)=x$, it means that $x$ is a local minimum point of $f$ in certain neighborhoods. Since $f$ is monotonically non-increasing under the filter function $T$, for any given $x\in \Bcal_n$, there exists a corresponding local minimum point by applying the filter function $T$ to $x$ for many times.
It indicates any point $x\in \Bcal_n$ can be classified by a certain local minimum point after repetitively applying the filter function. Consequently, we can divide the set $\mathcal{B}_n$ into subsets with respect to the classification of local minima. Let $X_1, X_2, ..., X_r$ be a partition of $\mathcal{B}_n$ such that for any $j\in \{1,\ldots,r\}$, every $x\in X_j$ has the same corresponding local minimum point. The following proposition provides a comparison between $q_\lambda$ and $\hat q_\lambda$.
\begin{proposition}
    \label{filterKL}
    If there exists some $x\in \mathcal B_n$ such that $p_\theta(x)>0$ and $f(x)>f(T(x))$, then for any sufficiently small $\lambda > 0$ satisfying
    \begin{equation}
        \label{filterKLcond}
        \mathbb E_{p_\theta}[f(x)-f(T(x))] \geq \lambda\log(\max_{1\leq i\leq r}|X_i|),
    \end{equation}
    it holds that
    \begin{equation*}
        \KL{p_\theta}{\hat q_\lambda}\leq\KL{p_\theta}{q_\lambda}.
    \end{equation*}

\end{proposition}
\begin{proof}
    Let $x_i=\argmin_{x'\in X_i} f(x')$ for $1\leq i\leq r$. Because of \eqref{eq:KL}, we only need to prove that
    \begin{equation*}
        \frac{1}{\lambda}\mathbb E_{p_{\theta}}[f(T(x))]+\log \hat Z_\lambda \leq \frac{1}{\lambda}\mathbb E_{p_{\theta}}[f(x)]+\log Z_\lambda.
    \end{equation*}
    Using (\ref{filterKLcond}), it suffices to prove that
    \begin{equation*}
        \log\hat Z_\lambda \leq \log(\max_{1\leq i\leq r}|X_i|)+\log Z_\lambda.
    \end{equation*}
    Thus, we only need to show that
    \begin{equation*}
        \sum_{x\in\Bcal_n}\exp\left(-\frac{f(T(x))}{\lambda}\right) \leq (\max_{1\leq i\leq r}|X_i|)\sum_{x\in\Bcal_n}\exp\left(-\frac{f(x)}{\lambda}\right).
    \end{equation*}
    It is true because for all $1\leq j\leq r$, we have
    \begin{equation*}
        \begin{split}
            \sum_{x\in X_j}\exp\left(-\frac{f(T(x))}{\lambda}\right) &\leq |X_j|\exp\left(-\frac{f(x_j)}{\lambda}\right)\\
            &\leq (\max_{1\leq i\leq r}|X_i|)\sum_{x\in X_j}\exp\left(-\frac{f(x)}{\lambda}\right).
        \end{split}
    \end{equation*}
    This completes the proof.
\end{proof}

The filter function $T(x)$ is induced by a given neighborhoods $\Ncal(x)$ under Definition \ref{filterdefinition}. Since it is hard to analyze the general case, we study  a randomly chosen neighborhood and show how it improves the quality of the solutions.  Note that the local search function $T_{LS}$  described in Section \ref{sec:local_search} is carefully designed.
Therefore, it is reasonable to believe that  $T_{LS}$  can also largely improve the performance of the algorithm.  For convenience, we denote $N=2^n$ and sort all possible points in $\Bcal_n=\{s_1,\dots,s_{N}\}$ such that $f(s_1) \le f(s_2) \le \cdots \le f(s_N)$.
We next show that the bounds of $f(T(x))$ and $\mathbb E_{p_\theta}[f(T(x))]$, for a large probability, are not related to samples $s_{M+1},s_{M+2},...,s_{N}$ for an integer $M$. When the neighborhood is large, then we do not need to consider most of bad samples.

\begin{lemma}
    Suppose that the cardinality  of each neighborhood $\mathcal N(s_i)$ is fixed to be $|\mathcal N(s_i)| \ge X\ge n+1$  and all elements in $\mathcal N(s_i)$ except $s_i$ are chosen uniformly at random from $\mathcal B_n\backslash\{s_i\}$.  For $\delta\in (0,1)$, let $M=\left\lceil\frac{\log (N/\delta)}{X-1}N\right\rceil+1$. Then, with probability at least $1-\delta$ over the choice of $T(x)$, it holds:
    \begin{enumerate}
        \item[1)] $f(T(x)) \in [f(s_1),f(s_M)],~\forall x\in \mathcal B_n$;
        \item[2)] $\mathbb E_{p_\theta}[f(T(x))]\leq \sum_{i=1}^{M-1}p_\theta(s_i)f(s_i)+ (1-\sum_{i=1}^{M-1}p_\theta(s_i))f(s_M)\leq f(s_M).$
    \end{enumerate}

    \label{lemma:bound}
\end{lemma}
\begin{proof}
    1) For all $1\leq i\leq N$, we have
    \begin{equation*}
        \begin{aligned}
            \mathbb P(f(T(s_i))>f(s_M))
            \leq & \mathbb P(\forall s\in \mathcal N(s_i)\backslash \{s_i\},f(s)>f(s_M)) \\
            \leq & \left(\frac{N-M}{N-1}\right)^{|\mathcal N(s_i)|-1}
            \leq \left(1-\frac{M-1}{N-1}\right)^{X-1}                                    \\
            \leq & \exp\left(-\frac{M-1}{N-1}(X-1)\right)\leq \frac{\delta}{N},
        \end{aligned}
        \label{lemma:bound eq}
    \end{equation*}
    where the second inequality is due to $\mathbb P(\cup_{ s}\{ f(s)>f(s_M)\})\leq \prod_{ s}\mathbb P( f(s)>f(s_M))$ and  $\mathbb P( f(s)>f(s_M))=(N-M)/(N-1)$ by the random choice of $s$, the third inequality uses $|\mathcal N(s_i)| \ge X$, the fourth inequality uses the fact $1-x\leq e^{-x}$ and the final relies on the choice of $M$. Then, it follows that
    \begin{equation*}
        \begin{aligned}
            \mathbb P(\forall x\in \mathcal B_n,f(T(x))\leq f(s_M))
            \geq 1-\sum_{i=1}^{N}\mathbb P(f(T(s_i))>f(s_M))
            \geq 1-\delta.
        \end{aligned}
    \end{equation*}

    2) For all $M\leq i\leq N$, we have
    \begin{equation*}
        \begin{aligned}
            \mathbb{P}(f(T(s_i))>f(s_M))
            \leq\frac{\delta}{N}.
        \end{aligned}
    \end{equation*}
    Hence, with probability no less than $1-\delta$ over the choice of $T$, we have
    \begin{equation*}
        f(T(s_i))\leq f(s_M),\quad \forall M\leq i\leq N.
    \end{equation*}
    Therefore,
    \begin{equation*}
        \begin{aligned}
            \mathbb E_{p_\theta}[f(T(x))] = & \sum_{i=1}^{N}p_\theta(s_i)f(T(s_i))                                          \\
            \leq                            & \sum_{i=1}^{M-1}p_\theta(s_i)f(s_i)+ (1-\sum_{i=1}^{M-1}p_\theta(s_i))f(s_M).
        \end{aligned}
    \end{equation*}
    This completes the proof.
\end{proof}

\subsection{Convergence analysis of SGD}

In this subsection, we will show that our algorithm converges to stationary points in expectation. It means that SGD with MCMC samples updates the policy effectively. We first make the following assumption on $p_{\theta}(x)$. For notational convenience, we write $\phi(x;\theta) = \log p_{\theta}(x)$ for any $x\in\Bcal_n$. Without specifically mentioned, the norm used in the context will be $\ell_2$ norm in Euclidean space.
\begin{assumption}
    \label{asm:SGD}
    For any $x\in \Bcal_n$, $\phi(x;\theta)$ is differentiable with respect to $\theta$. Moreover, there exists constants $M_1,M_2,M_3>0$ such that, for any $x\in\Bcal_n$,
    \begin{enumerate}
        \item  $\sup_{\theta\in\real^n}|\phi(x;\theta)|\le M_1$,
        \item $\sup_{\theta\in\real^n}\norm{\nabla_\theta\phi(x;\theta)}\le M_2$,
        \item $\norm{\nabla_{\theta}\phi(x;\theta^1)-\nabla_{\theta}\phi(x;\theta^2)}\le M_3\norm{\theta^1-\theta^2},~\forall \theta^1,\theta^2\in\real^n$.
    \end{enumerate}
\end{assumption}
The above standard assumptions on the smoothness of the policy $p_{\theta}$ are adopted in several literatures on the convergence analysis of policy gradient algorithms. These can be satisfied in practice by our mean field sampling policy and neural network.

Moreover, we estimate the stochastic gradient with MCMC samples but not direct sampling. The difference is that MCMC sampling is not independent or unbiased. To explain the effectivity of MCMC, we introduce the following definitions on the finite-state time-homogenous Markov chain.

\begin{definition}[spectral gap]
    Let $P$ be the transition matrix of a finite-state time homogenous Markov chain. Then,  the spectral gap of $P$ is defined as
    \begin{equation*}
        \begin{aligned}
            \gamma(P):=\frac{1-\max\{|\lambda_2(P)|,|\lambda_{2^{n}}(P)|\}}{2}\in (0,1],
        \end{aligned}
    \end{equation*}
    where $\lambda_{i}(P)$ is the $i$-th largest eigenvalue of the matrix $P$.
\end{definition}

The spectral gap of a Markov transition matrix determines the ergodicity of the corresponding Markov chain. A positive spectral gap guarantees that the samples generated by the MCMC algorithm can approximate the desired distribution. For any non-stationary Markov chain, it will take a long time to mix until the distribution of its current state becomes sufficiently close to its stationary distribution. It leads to bias in gradient estimation. However, we can use the spectral gap to bound the bias and variance. We warm up the Markov chain and discard the beginning $m_0$ samples, leaving $m$ samples for estimating the stochastic gradient. The follow lemma gives error bounds of the stochastic gradient with MCMC samples.
\begin{lemma}
    \label{lem:error}
    Let Assumption \ref{asm:SGD} hold and $S=\{s_i^{j}\}_{1\leq i\leq k,~1\leq j\leq m}$ be the samples generated from the MH algorithm. Suppose that the Markov chains have the stationary distribution $p_{\theta}$ and the initial distribution $\nu$.
    Let $\gamma=\inf_{\theta\in\mathbb{R}}\gamma(P_{\theta})\in (0,1]$ be the infimum of spectral gaps of all Markov transition matrices $P_{\theta}$. Then, the sampling error of the stochastic gradient $\bar g_\lambda(\theta,S)$ is given by
    \begin{equation*}
        \begin{aligned}
            \norm{\ex{\nu}{\bar g_\lambda(\theta,S)}-\nabla_\theta L_\lambda(\theta)}   & \leq \frac{\chi M_2(B(f\circ T)+\lambda M_1)}{m\gamma}=:B(m),         \\
            \ex{\nu}{\norm{\bar g_\lambda(\theta,S)-\nabla_\theta L_\lambda(\theta)}^2} & \leq \frac{8nC(k)M_2^2(B(f\circ T)+\lambda M_1)^2}{mk\gamma}=:V(m,k),
        \end{aligned}
    \end{equation*}
    where $\chi=\chi(\nu,p_{\theta}):=\Ebb_{p_\theta}\left[\left(\frac{d\nu}{dp_{\theta}}-1\right)^2\right]^{1/2}$ is the square-chi divergence between $\nu$ and $p_{\theta}$ and $C(k)=1+k\log(1+\chi)$.
    \label{samplingerror}
\end{lemma}
\begin{proof}
    We first prove that the gradient $\nabla_\theta L_\lambda(\theta)$ defined by \eqref{eq:gradient} is bounded. As Assumption \ref{asm:SGD} holds, $f(T(x)), \log p_{\theta}(x), \norm{\nabla_{\theta}\log p_{\theta}(x)}$ are bounded by $B(f\circ T), M_1, M_2$ respectively. Hence, it is easily derived that $\norm{A_{\lambda}(x;\theta)\nabla_{\theta}\log p_{\theta}(x)}$ is bounded by $(B(f\circ T)+\lambda M_1)M_2$.

    We now  estimate  the error of  MCMC for an arbitrary function $h$ bounded by $M$. The empirical distribution of $s_{i}^{j}$ can be denoted by $\nu_j=\nu P_{\theta}^{j-1}$ for any batch $i$. Then, it follows from the Cauchy-Schwarz inequality that
    \begin{equation}
        \label{eq:change-distribution}
        \begin{aligned}
            \abs{\Ebb_{\nu_{j}}[h]-\Ebb_{p_{\theta}}[h]}\leq \Ebb_{p_\theta}\left[\left(\frac{d\nu_i}{dp_{\theta}}-1\right)^2\right]^{1/2}\Ebb_{p_\theta}\left[h(x)^2\right]^{1/2}\leq M\chi(\nu_j,p_{\theta}).
        \end{aligned}
    \end{equation}
    Since $\nu_j=\nu_{j-1}P_{\theta}$, it also holds that
    \begin{equation}
        \label{eq:j-distribution}
        \begin{aligned}
            \chi(\nu_j,p_{\theta}) & =\Ebb_{p_\theta}\left[\left(\frac{d\nu_j}{dp_{\theta}}-1\right)^2\right]^{1/2}=\Ebb_{p_\theta}\left[\left(\frac{d(\nu_{j-1}P_{\theta})}{dp_{\theta}}-1\right)^2\right]^{k/2} \\
                                   & \leq (1-\gamma)\chi(\nu_{j-1},p_{\theta})\leq (1-\gamma)^{j-1}\chi.
        \end{aligned}
    \end{equation}
    Hence, we denote $e_{k,m}(h)=\frac{1}{mk}\sum_{1\leq i\leq k,1\leq j \leq m} h(s^{j}_i)$ and obtain
    \begin{equation*}
        \begin{aligned}
            \abs{\Ebb_{\nu}\left[e_{k,m}(h)\right]-\Ebb_{p_{\theta}}[h]} & \leq \frac{1}{k}\sum_{i=1}^{k}\abs{\Ebb_{\nu}\left[\frac{1}{mk}\sum_{1\leq i\leq k,1\leq j \leq m} h(s^{j}_i)\right]-\Ebb_{p_{\theta}}[h]} \\
                                                                         & \leq\frac{1}{mk}\sum_{1\leq i\leq k,1\leq j \leq m}M\chi(\nu_j,p_{\theta})                                                                 \\
                                                                         & \leq \frac{1}{mk}\sum_{1\leq i\leq k,1\leq j \leq m}(1-\gamma)^{j-1}M\chi\leq \frac{M\chi}{m\gamma},
        \end{aligned}
    \end{equation*}
    where the first inequality holds from \eqref{eq:change-distribution} and the second is due to \eqref{eq:j-distribution}.

    Using  the following concentration inequality from Fan et al. \cite{fan2021hoeffding}  for the general MCMC gives
    \begin{equation*}
        \Pbb_{\nu}\left(e_{m,k}(h)-\Ebb_{p_{\theta}}[h]>\epsilon\right)\leq (1+\chi)^{k}\exp\left( -\frac{mk\gamma\epsilon^2}{4M^2}\right).
    \end{equation*}
    Then, we can obtain the variance bound
    \begin{equation}
        \label{eq:variance}
        \begin{aligned}
            \ex{\nu}{\abs{e_{k,m}(h)-\Ebb_{p_{\theta}}[h]}^2}
            =    & \int_{0}^{+\infty}2s \Pbb_{\nu}\left(\abs{e_{m,k}(h)-\Ebb_{p_{\theta}}[h]}>s\right)ds \\
            \leq & \frac{8\left(1+k\log(1+\chi)\right)M^2}{mk\gamma}.
        \end{aligned}
    \end{equation}
    For any vector $v\in \R^n$, taking $h=v^{T}(A_{\lambda}\nabla_{\theta}\log p_{\theta})$ yields
    \begin{equation}
        \begin{aligned}
            \norm{\ex{\nu}{\bar g_\lambda(\theta)}-\nabla_\theta L_\lambda(\theta)} & =\max_{\norm{v}=1} \abs{v^{T}\left(\ex{\nu}{\bar g_\lambda(\theta)}-\nabla_\theta L_\lambda(\theta)\right)} \\
                                                                                    & \leq \frac{\chi M_2(B(f\circ T)+\lambda M_1)}{m\gamma},
        \end{aligned}
    \end{equation}
    where we use the fact that $\abs{v^{T}\nabla_\theta L_\lambda(\theta)}\leq\norm{v}\norm{\nabla_\theta L_\lambda(\theta)} \leq (B(f\circ T)+\lambda M_1)M_2$ discussed in the first paragraph of this proof. For the variance, it follows from \eqref{eq:variance} that
    \begin{equation}
        \begin{aligned}
            \ex{\nu}{\norm{\bar g_\lambda(\theta,S)-\nabla_\theta L_\lambda(\theta)}^2}
             & \leq n \max_{\norm{v}=1}\ex{\nu}{\abs{v^{T}\bar g_\lambda(\theta,S)-v^{T}\nabla_\theta L_\lambda(\theta)}^2} \\
             & \leq \frac{8n\left(1+k\log(1+\chi)\right)M_2^2(B(f\circ T)+\lambda M_1)^2}{mk\gamma}.
        \end{aligned}
    \end{equation}
    This completes the proof.
\end{proof}

The bias $B(m)$ has an order of $O\left(\frac{1}{m}\right)$ and the variance $V(m,k)$ is of $O\left(\frac{1}{mk}\right)$. The MH algorithms ensures enough locality and randomness in exploring $\Bcal_n$. These error bounds show how hyper-parameters influence the stochastic gradient.

Under Assumption \ref{asm:SGD}, we can prove the $l$-smoothness of the cost function $L_{\lambda}(\theta)$ and perform a standardized analysis in stochastic optimization. The $l$-smoothness of the loss function is shown in the following lemma.
\begin{lemma}
    Let Assumption \ref{asm:SGD} hold.
    There exists a constant $l>0$ such that $ L_\lambda(\theta)$ is $l$-smooth, that is
    \begin{equation}
        \label{eq:Lips}
        \norm{\nabla_\theta L_\lambda(\theta^1)-\nabla_\theta L_\lambda(\theta^2)}\leq l\norm{\theta^1-\theta^2},~\forall \theta^1,\theta^2 \in \mathbb{R}^{d}.
    \end{equation}
\end{lemma}

\begin{proof}
    For convenience, we denote that $\phi_i(x)=\phi(x;\theta^i)$ and $\nabla\phi_i(x)=\nabla_{\theta}\phi(x;\theta^i)$ for $i=1,2$. Under Assumption \ref{asm:SGD}, it holds that for any $x\in \Bcal_n$,
    \begin{equation*}
        \begin{aligned}
            \abs{\phi_i(x)}\leq M_1,                                   & \quad \norm{\nabla\phi_i(x)}\leq M_2 ,\quad \mathrm{for}~ i=1,2,              \\
            \abs{\phi_1(x)-\phi_2(x)}\leq M_2\norm{\theta^1-\theta^2}, & \quad \norm{\nabla\phi_1(x)-\nabla\phi_2(x)}\leq M_3\norm{\theta^1-\theta^2}.
        \end{aligned}
    \end{equation*}
    Furthermore, using the Cauchy mean value theorem, we have that there exists $\lambda \in (0,1)$ such that $e^{x_1}-e^{x_2}=e^{\lambda x_1+(1-\lambda)x_2}(x_1-x_2)$ for arbitrary $x_1,x_2$. This fact leads to
    \begin{equation*}
        \begin{aligned}
            \abs{e^{\phi_1(x)}-e^{\phi_2(x)}} & =\abs{e^{\lambda\phi_1(x)+(1-\lambda)\phi_2(x)}}\abs{\phi_1(x)-\phi_2(x)} \\
                                              & \leq \left(e^{\phi_1(x)}+e^{\phi_2(x)}\right)M_2\norm{\theta^1-\theta^2}.
        \end{aligned}
    \end{equation*}

    Because $p_{\theta}(x)=e^{\phi(x;\theta)}$, we can derive that
    \begin{equation*}
        \begin{aligned}
              & \nabla_\theta L_\lambda(\theta^1)-\nabla_\theta L_\lambda(\theta^2)                                     \\=&\sum_{x\in \Bcal_n}\left(e^{\phi_1(x)}(f(T(x))+\lambda\phi_1(x))\nabla\phi_1(x)-e^{\phi_2(x)}(f(T(x))+\lambda\phi_2(x))\nabla\phi_2(x)\right)\\
            = & \sum_{x\in \Bcal_n}\left(e^{\phi_1(x)}-e^{\phi_2(x)}\right)(f(T(x))+\lambda\phi_1(x))\nabla\phi_1(x)    \\
              & +\sum_{x\in \Bcal_n}e^{\phi_2(x)}(f(T(x))+\lambda\phi_1(x))\left(\nabla\phi_1(x)-\nabla\phi_2(x)\right) \\
              & +\lambda\sum_{x\in \Bcal_n}e^{\phi_2(x)}(\phi_1(x)-\phi_2(x))\nabla\phi_2(x).
        \end{aligned}
    \end{equation*}
    Hence, it follows that
    \begin{equation*}
        \begin{aligned}
                 & \norm{\nabla_\theta L_\lambda(\theta^1)-\nabla_\theta L_\lambda(\theta^2)}                                                                                   \\
            \leq & \sum_{x\in \Bcal_n}\left(e^{\phi_1(x)}+e^{\phi_2(x)}\right)M_2\norm{\theta^1-\theta^2}(B(f\circ T)+\lambda M_1)M_2                                           \\
                 & + \sum_{x\in \Bcal_n}e^{\phi_2(x)}(B(f\circ T)+\lambda M_1)M_3\norm{\theta^1-\theta^2}+\lambda\sum_{x\in \Bcal_n}e^{\phi_2(x)} M_2^2\norm{\theta^1-\theta^2} \\
            =    & \left(2M_2^2(B(f\circ T)+\lambda M_1)+\lambda M_2^2+M_3(B(f\circ T)+\lambda M_1)\right)\norm{\theta^1-\theta^2}.
        \end{aligned}
    \end{equation*}
    Let $l=2M_2^2(B(f\circ T)+\lambda M_1)+\lambda M_2^2+M_3(B(f\circ T)+\lambda M_1)$ and then \eqref{eq:Lips} holds.
\end{proof}

It follows from \eqref{eq:Lips} that
\begin{equation}
    \label{eq:Lips2}
    L_{\lambda}(\theta^2)\leq L_{\lambda}(\theta^1)+\left\langle\nabla_{\theta}L_{\lambda}(\theta^1),\theta^2-
    \theta^1\right\rangle+\frac{l}{2}\norm{\theta^2-
        \theta^1}^2.
\end{equation}
Under our discussion, Algorithm \ref{alg:MCPG} updates the policy with the stochastic gradient $\bar{g}(\theta)$. We have known the error bound of $\bar{g}(\theta)$ and $l$-smoothness of the cost function. By discussing the expected descent within each iteration, we can establish the  first-order convergence  and give the convergence rate as follows.
\begin{theorem}
    Let Assumption \ref{asm:SGD} hold and $\{\theta^t\}$ be generated by MCPG. If the stepsize satisfies $\eta^t\leq \frac{1}{2L}$, then for any $\tau$, we have

    \begin{equation}
        \label{eq:SGDconv}
        \min_{1\leq t\leq \tau}\ex{}{\norm{\nabla_\theta L_\lambda(\theta^t)}^2}\leq O\left(\frac{1}{\sum_{t=1}^{\tau}\eta^t}+\frac{\sum_{t=1}^{\tau}(\eta^{t})^2}{mk\sum_{t=1}^{\tau}\eta^t}+\frac{1}{m^2}\right).
    \end{equation}
    In particular, if the stepsize is chosen as $\eta^t=\frac{c\sqrt{mk}}{\sqrt{t}}$ with $c\leq \frac{1}{2l}$, then we have
    \begin{equation}
        \min_{1\leq t\leq \tau}\ex{}{\norm{\nabla_\theta L_\lambda(\theta^t)}^2}\leq O\left(\frac{\log \tau}{\sqrt{mk\tau}}+\frac{1}{m^2}\right).
    \end{equation}
\end{theorem}
\begin{proof}
    The MCPG updates with the stochastic gradient as $\theta^{t+1}=\theta^{t}-\eta^{t}\bar{g}_{\lambda}(\theta^t)$. For convenience, we denote that $L^t=L_{\lambda}(\theta^t)$, $g^t=\nabla_{\theta}L_{\lambda}(\theta^t)$ and $\hat{g}^t=\bar{g}_{\lambda}(\theta^t)$ for any $t\geq 1$.Using the Lipschitz condition of $L_{\lambda}$, it follows from \eqref{eq:Lips2} that
    \begin{equation}
        \label{eq:descent1}
        \begin{aligned}
            L^{t+1} & \leq L^{t}+\left\langle g^{t},\theta^{t+1}-
            \theta^{t}\right\rangle+\frac{l}{2}\norm{\theta^{t+1}-
                \theta^{t}}^2
            \\
                    & =L^{t}-\eta^{t}\left\langle g^{t},\hat{g}^{t}\right\rangle+\frac{\eta^{t2} l}{2}\norm{\hat{g}^{t}}^2                                                                                      \\
                    & =L^{t}-(\eta^{t}-\eta^{t2} l)\left\langle g^{t},\hat{g}^{t}-g^{t}\right\rangle-\left(\eta^{t}-\frac{\eta^{t2} l}{2}\right)\norm{g^{t}}^2+\frac{\eta^{t2} l}{2}\norm{\hat{g}^{t}-g^{t}}^2.
        \end{aligned}
    \end{equation}
    Taking the condition expectation of \eqref{eq:descent1} over $\theta^{t}$, we obtain that
    \begin{equation}
        \label{eq:descent2}
        \begin{aligned}
            \ex{}{L^{t+1}|\theta^{t}}\leq & \ex{}{L^{t}|\theta^{t}}-(\eta^{t}-\eta^{t2} l)\left\langle g^{t},\ex{}{\hat{g}^{t}|\theta^{t}}-g^{t}\right\rangle            \\
                                          & -\left(\eta^{t}-\frac{\eta^{t3}l}{2}\right)\norm{g^{t}}^2+\frac{\eta^{t2} l}{2}\ex{}{\norm{\hat{g}^{t}-g^{t}}^2|\theta^{t}}.
        \end{aligned}
    \end{equation}
    Using the fact $\abs{\left\langle x,y\right\rangle}\leq (\norm{x}^2+\norm{y}^2)/2$, we have
    \begin{equation}
        \label{eq:descent3}
        -(\eta^{t}-\eta^{t2}l)\left\langle g^{t},\ex{}{\hat{g}^{t}|\theta^{t}}-g^{t}\right\rangle\leq \frac{\eta^{t}-\eta^{t2}l}{2}\norm{g^{t}}^2+\frac{\eta^{t}}{2}\norm{\ex{}{\hat{g}^{t}|\theta^{t}}-g^{t}}^2.
    \end{equation}
    Therefore, we plug in \eqref{eq:descent3} to \eqref{eq:descent2} and then get
    \begin{equation}
        \label{eq:descent4}
        \begin{aligned}
                 & 2(\ex{}{L^{t}|\theta^{t}}-\ex{}{L^{t+1}|\theta^{t}})                                                                                           \\
            \geq & \eta^{t}\norm{g^{t}}^2-\eta^{t}\cdot\norm{\ex{}{\hat{g}^{t}|\theta^{t}}-g^{t}}^2-\eta^{t2} l\cdot\ex{}{\norm{ \hat{g}^{t}-g^{t}}^2|\theta^{t}} \\
            \geq & \eta^{t}\norm{g^{t}}^2-\eta^{t} \cdot B^2(m)- \eta^{t2} l\cdot V(m,k),
        \end{aligned}
    \end{equation}
    where the second inequality is due to Lemma \ref{lem:error}.
    Since $ \ex{}{L^{t}|\theta^{t}}=\ex{}{L^{t}|\theta^{t-1}}$,  summing \eqref{eq:descent4} over $t=1,\dots,\tau$ and taking the total expectation  yields
    \begin{equation}
        \label{eq:descent5}
        \begin{aligned}
                 & \left(\sum_{t=1}^{\tau}\eta^{t}\right)\min_{1\leq t\leq \tau}\ex{}{\norm{g^{t}}^2}\leq\sum_{t=1}^{\tau}\eta^{t} \cdot \ex{}{\norm{g^{t}}^2} \\ \leq& 2(L^1-\ex{}{L^{\tau+1}|\theta^{\tau}})+\sum_{t=1}^{\tau}\eta^{t2}\cdot lV(m,k)+\sum_{t=1}^{\tau}\eta^{t} \cdot B^2(m)\\
            \leq & O\left(1\right)+\sum_{t=1}^{\tau}\eta^{t2} \cdot O\left(\frac{1}{mk}\right)+\sum_{t=1}^{\tau}\eta^{t} \cdot O\left(\frac{1}{m^2}\right).
        \end{aligned}
    \end{equation}
    Dividing both sides by $\sum_{t=1}^{\tau}\eta^{t}$ gives the inequality \eqref{eq:SGDconv}.
\end{proof}

\section{Numerical results}
\label{sec:numerical}
In this section, we demonstrate the effectiveness of our algorithms on a wide variety of test problems, including MaxCut, cheeger cut, binary quadratic programming, MIMO detection problem and MaxSAT.

\subsection{Experimental settings}

We apply mean field approximation  and express the distribution $p_\theta(x)$ by  \eqref{eq:mf}. For each variable $x_i$, there is a parameter $\theta_i$ in the probabilistic model. A sigmoid function is then applied to convert the output to the probability within $(0,1)$. In order to promote the sampling diversity, the probability is scaled to the range $(\alpha, 1-\alpha)$ for a given constant $0<\alpha<0.5$. Altogether, the parameter $\mu_i$ in \eqref{eq:mf}  is given as
\begin{equation} \label{mu-sigmoid}
    \mu_i = \phi_i(\theta_i) =\frac{1-2\alpha}{1+\exp(-\theta_i)}+\alpha,\quad 1\leq i\leq n.
\end{equation}
In practice, we take $\alpha=0.2$.  Since a large number of samples are required in Algorithm \ref{alg:MCPG}, the simple scheme \eqref{eq:mf} with \eqref{mu-sigmoid} is very fast on GPU and its performance is quite competitive to parameterization based on a few typical graph neural network in our experiments.

We compare MCPG against various methods, including state-of-the-art solvers, across all test problems. Note that MCPG utilizes the same computational device as the comparative algorithms, but additionally takes advantage of GPU acceleration.
Despite MCPG described in Algorithm \ref{alg:MCPG}, we consider the following two variants of our algorithm:
\begin{itemize}
    \item MCPG-U: In Algorithm \ref{alg:MCPG}, the sampling set $S$ is generated from the parameterized policy distribution $p_\theta(x)$. On the contrast, MCPG-U uses the fixed vertice-wise distribution $\mu_i = 0.5$. Therefore, the distribution of MCPG-U is equivalent to the case when the regular factor $\lambda$ tends to infinity. The comparison between MCPG-U and Algorithm \ref{alg:MCPG} illustrates the effectiveness of the parameterized policy distribution.
    \item MCPG-P:
          Existing unsupervised learning methods\cite{karalias2020erdos} use the expectation of the cost function $\mathbb{E}[f(x)]$ as the loss function to train a learning model. We notice that our sampling methods along with the filter scheme can be used as a powerful post-processing procedure. Therefore, in MCPG-P, we follow these existing learning methods to optimize the policy distribution at the beginning by minimizing $\mathbb{E}[f(x)]$, and then use this distribution instead of $p_\theta(x)$ to generate samples.
\end{itemize}
Our code
is available at \url{https://github.com/optsuite/MCPG}.

\subsection{MaxCut}
The MaxCut problem aims to divide a given weighted graph $G = (V,E)$ into two parts and maximize the total weight of the edges connecting two parts. This problem can be expressed as a binary programming problem:
\begin{equation}
    \label{prob:primary MaxCut_2}
    \max  \quad  \sum_{(i,j) \in E} w_{ij} (1-x_i x_j), \quad
    \st   \quad  x\in \{-1, 1\}^n.
\end{equation}
In this subsection, we show the effectiveness of our proposed algorithms on the MaxCut problems on both the Gset dataset and randomly generated graphs.

\subsubsection{Numerical comparison on Gset instances}
We performs experiments on standard MaxCut benchmark instances based on the publicly available Gset dataset commonly used for testing MaxCut algorithms. We select 7 graphs from Gset instances including two Erdos-Renyi graphs with uniform edge probability, two graphs where the connectivity gradually decays from node $1$ to $n$, two $4$-regular toroidal graphs, and one of the largest Gset instances.

The algorithms listed in the comparison include Breakout Local Search (BLS)\cite{benlic2013breakout}, Extremal Optimization (EO)\cite{boettcher2001extremal}, DSDP\cite{choi2000solving}, EMADM\cite{liu2022entropy}, RUN-CSP\cite{toenshoff2019run} and PI-GNN\cite{schuetz2022combinatorial}.
BLS is a heuristic algorithm combining local search and adaptive perturbation, which provides most of the best known solutions for the Gset dataset. EO is a heuristic algorithm inspired by equilibrium statistical physics and is widely applied on various graph bipartitioning problems. DSDP is an SDP solver based on dual-scaling algorithm. EMADM is the another SDP-based algorithm particularly designed for binary optimization, which preserves the rank-one constraint in order to find high quality binary solutions. RUN-CSP is a recurrent GNN architecture for maximum constraint satisfaction problems. PI-GNN is an unsupervised learning methods to solve graph partitioning problems based on graph neural networks. 

For BLS, DSDP, RUN-CSP and PI-GNN, we directly use the results presented in the referenced papers due to two reasons. First, the performance of all these algorithms is highly dependent on the implementation. Second, the stochastic nature of heuristic techniques inevitably leads to the deviation between the results  presented in the papers and the reproduction experiments. Thus we regard the numerical experiments reported in the papers as the most reliable results.

\begin{table}[htbp]
    \centering
    \caption{Computational results on selected graphs with results sourced from references.}
    \setlength{\tabcolsep}{3pt}
    \renewcommand\arraystretch{1.2}
    \begin{tabular}{|l|l|l|l|l|l|l|l|l|l|}
        \hline
        Graph & Nodes  & Edges  & MCPG            & BLS             & DSDP           & RUN-CSP        & PI-GNN & EO            & EMADM         \\ \hline
        G14   & 800    & 4,694  & \textbf{3,064}  & \textbf{3,064}  & 2,922          & 2,943          & 3,026  & 3047          & 3045          \\ \hline
        G15   & 800    & 4,661  & \textbf{3,050}  & \textbf{3,050}  & 2,938          & 2,928          & 2,990  & 3028          & 3034          \\ \hline
        G22   & 2,000  & 19,990 & \textbf{13,359} & \textbf{13,359} & 12,960         & 13,028         & 13,181 & 13215         & 13297         \\ \hline
        G49   & 3,000  & 6,000  & \textbf{6,000}  & \textbf{6,000}  & \textbf{6,000} & \textbf{6,000} & 5,918  & \textbf{6000} & \textbf{6000} \\ \hline
        G50   & 3,000  & 6,000  & \textbf{5,880}  & \textbf{5,880}  & \textbf{5,880} & \textbf{5,880} & 5,820  & 5878          & 5870          \\ \hline
        G55   & 5,000  & 12,468 & \textbf{10,296} & 10,294          & 9,960          & 10,116         & 10,138 & 10107         & 10208         \\ \hline
        G70   & 10,000 & 9,999  & \textbf{9595}   & 9,541           & 9,456          & -              & 9,421  & 8513          & 9557          \\ \hline
    \end{tabular}
    \label{Gset 1}
\end{table}

Table \ref{Gset 1} reports the best results achieved by MCPG and the comparison algorithms. It should be noticed that MCPG finds the best results for all the graphs. Compared to PI-GNN, an unsupervised learning framework, MCPG is more stable in finding the best solutions due to its efficient sampling methods and local search techniques. It is worth mentioning that the cut size $9595$ for G70 is the best-known MaxCut and is much better than previous reported results.

\subsubsection{Numerical results on Gset instances within limited time}
In this section, we control the running time for MCPG and report the numerical results on Gset instances. Note that some results are worse than that in Table \ref{Gset 1}, because results in Table \ref{Gset 1} are achieved by running the algorithm with about $3-5$ times of the controlled time, and by repeating many times. We choose a heuristic-based algorithm EO and an SDP-based algorithm EMADM as the comparison algorithms. As shown in Table \ref{Gset 1}, these two algorithms achieve fairly good performance compared to other algorithms, thus we do not  include other algorithms in the comparison. On the other hand, the performance of both of the two algorithms is influenced by only several hyper-parameters, and the results of the reproduction experiments are barely consistent with the results presented in the references. We describe the implementation details of the algorithms utilized in the comparison.
\begin{itemize}
    \item EO: The standard version of EO requires $n^2$ iterations for the search procedure, where $n$ is the number of nodes. With this setting, the time consumption is unaffordable on large graphs, such as G81. Therefore, we set the number of iterations to $\min\{n^2, 1000n\}$ to make the time consumption comparable to MCPG.
    \item EMADM: The penalized SDP relaxation is solved by a first-order type method. However, the quality of the solution given by EMADM is affected by the initial point. Therefore, we report the best result of three repeated experiments as the gap of EMADM, and the time reported is the summation of the experiments.
\end{itemize}

We use the best-known results as the benchmark. Most of the best-known results are reported by BLS, except for G55 and G70,  whose optimal result is found by MCPG with a long running time. Denoting $\mathrm{UB}$ as the best-known results and $\mathrm{obj}$ as the cut size, the reported gap is defined as follows:
$$
    \mathrm{gap} = \frac{\mathrm{UB} - \mathrm{obj}}{\mathrm{UB}} \times 100\%.
$$

\begin{table*}
    \centering

    \setlength{\tabcolsep}{1.5pt}
    \caption{Detailed result on selected Gset instances. UB is the result reported by BLS, except for G55 and G70, where we find better solutions after running MCPG for a long time.}
    \begin{tabular}{lllllllllllllll}
        \toprule
        \multicolumn{2}{c}{Problem} & \multicolumn{3}{c}{MCPG} & \multicolumn{3}{c}{MCPG-U} & \multicolumn{3}{c}{MCPG-P} & \multicolumn{2}{c}{EO}  & \multicolumn{2}{c}{EMADM}                                                                                               \\
        \cmidrule(r){1-2} \cmidrule(r){3-5} \cmidrule(r){6-8} \cmidrule(r){9-11} \cmidrule(r){12-13} \cmidrule(r){14-15}
        name                        & UB                       & \multicolumn{2}{l}{gap}    & time                       & \multicolumn{2}{l}{gap} & time                      & \multicolumn{2}{l}{gap} & time & gap    & time  & gap & time                                \\

                                    &                          & (best,                     & mean)                      &                         & (best,                    & mean)                   &      & (best, & mean) &     & mean  &      &               &      \\
        \cmidrule(r){1-2} \cmidrule(r){3-5} \cmidrule(r){6-8} \cmidrule(r){9-11} \cmidrule(r){12-13} \cmidrule(r){14-15}
        G22                         & 13359                    & \textbf{0.01}              & 0.10                       & 55                      & 0.01                      & 0.13                    & 51   & 0.64   & 0.76  & 84  & 1.21  & 116  & 0.42          & 66   \\ \hline
        G23                         & 13344                    & \textbf{0.02}              & 0.15                       & 56                      & 0.07                      & 0.18                    & 51   & 0.50   & 0.61  & 85  & 0.95  & 116  & 0.39          & 66   \\ \hline
        G24                         & 13337                    & \textbf{0.03}              & 0.18                       & 56                      & 0.04                      & 0.15                    & 51   & 0.31   & 0.48  & 85  & 0.97  & 116  & 0.32          & 60   \\ \hline
        G25                         & 13340                    & \textbf{0.07}              & 0.18                       & 55                      & 0.08                      & 0.18                    & 51   & 0.44   & 0.54  & 85  & 0.95  & 115  & 0.45          & 66   \\ \hline
        G26                         & 13328                    & \textbf{0.08}              & 0.17                       & 55                      & 0.10                      & 0.17                    & 50   & 0.41   & 0.53  & 84  & 0.92  & 115  & 0.36          & 66   \\ \hline
        G27                         & 3341                     & \textbf{0.15}              & 0.61                       & 55                      & 0.24                      & 0.78                    & 50   & 1.23   & 1.82  & 75  & 4.13  & 116  & 1.12          & 72   \\ \hline
        G28                         & 3298                     & 0.15                       & 0.55                       & 55                      & \textbf{0.12}             & 0.67                    & 50   & 0.61   & 1.29  & 76  & 4.13  & 116  & 1.66          & 72   \\ \hline
        G29                         & 3405                     & \textbf{0.00}              & 0.83                       & 55                      & 0.56                      & 0.92                    & 51   & 1.59   & 2.14  & 86  & 4.04  & 116  & 2.10          & 72   \\ \hline
        G30                         & 3412                     & \textbf{0.06}              & 0.44                       & 55                      & 0.09                      & 0.59                    & 50   & 1.55   & 1.98  & 86  & 3.94  & 115  & 1.29          & 66   \\ \hline
        G31                         & 3309                     & \textbf{0.12}              & 0.58                       & 55                      & 0.30                      & 0.67                    & 50   & 0.73   & 1.54  & 78  & 3.98  & 115  & 1.28          & 66   \\ \hline
        G32                         & 1410                     & \textbf{0.57}              & 1.01                       & 54                      & 1.56                      & 1.83                    & 49   & 2.41   & 2.85  & 73  & 2.67  & 99   & 2.01          & 66   \\ \hline
        G33                         & 1382                     & \textbf{0.72}              & 1.03                       & 53                      & 1.59                      & 2.11                    & 49   & 1.74   & 1.94  & 60  & 2.51  & 99   & 1.57          & 66   \\ \hline
        G34                         & 1384                     & \textbf{0.58}              & 0.92                       & 54                      & 1.16                      & 1.60                    & 49   & 1.16   & 1.59  & 57  & 2.26  & 99   & 1.83          & 66   \\ \hline
        G35                         & 7684                     & \textbf{0.26}              & 0.42                       & 55                      & 0.42                      & 0.51                    & 49   & 0.83   & 0.94  & 79  & 0.79  & 107  & 0.61          & 66   \\ \hline
        G36                         & 7678                     & \textbf{0.31}              & 0.46                       & 55                      & 0.40                      & 0.52                    & 50   & 0.86   & 0.91  & 79  & 0.79  & 107  & 0.69          & 66   \\ \hline
        G37                         & 7689                     & \textbf{0.23}              & 0.42                       & 54                      & 0.38                      & 0.50                    & 50   & 0.74   & 0.82  & 79  & 0.78  & 107  & 0.73          & 66   \\ \hline
        G38                         & 7687                     & \textbf{0.22}              & 0.41                       & 54                      & 0.36                      & 0.52                    & 50   & 0.94   & 1.04  & 78  & 0.83  & 107  & 0.67          & 72   \\ \hline
        G39                         & 2408                     & \textbf{0.08}              & 0.66                       & 54                      & 0.21                      & 0.76                    & 50   & 1.70   & 2.25  & 77  & 3.78  & 106  & 1.94          & 78   \\ \hline
        G40                         & 2400                     & \textbf{0.04}              & 0.44                       & 54                      & 0.25                      & 0.61                    & 50   & 0.50   & 1.04  & 65  & 4.11  & 106  & 1.88          & 84   \\ \hline
        G41                         & 2405                     & \textbf{0.00}              & 0.37                       & 54                      & 0.08                      & 0.58                    & 50   & 2.12   & 2.91  & 80  & 4.46  & 106  & 1.93          & 84   \\ \hline
        G42                         & 2481                     & \textbf{0.00}              & 0.54                       & 54                      & 0.28                      & 0.77                    & 50   & 2.10   & 2.81  & 80  & 4.21  & 106  & 2.70          & 84   \\ \hline
        G43                         & 6660                     & \textbf{0.00}              & 0.04                       & 28                      & 0.00                      & 0.05                    & 26   & 0.32   & 0.42  & 45  & 0.96  & 45   & 0.32          & 24   \\ \hline
        G44                         & 6650                     & \textbf{0.00}              & 0.01                       & 28                      & 0.00                      & 0.05                    & 25   & 0.35   & 0.47  & 46  & 0.97  & 45   & 0.41          & 30   \\ \hline
        G45                         & 6654                     & \textbf{0.00}              & 0.08                       & 28                      & 0.00                      & 0.09                    & 26   & 0.20   & 0.42  & 45  & 0.85  & 45   & 0.46          & 30   \\ \hline
        G46                         & 6649                     & \textbf{0.00}              & 0.07                       & 28                      & 0.00                      & 0.08                    & 25   & 0.30   & 0.45  & 45  & 0.80  & 45   & 0.50          & 30   \\ \hline
        G47                         & 6657                     & \textbf{0.00}              & 0.06                       & 28                      & 0.00                      & 0.08                    & 26   & 0.23   & 0.41  & 45  & 0.82  & 45   & 0.33          & 30   \\ \hline
        G55                         & 10296                    & 0.34                       & 0.47                       & 145                     & \textbf{0.23}             & 0.36                    & 138  & 1.08   & 1.45  & 151 & 2.75  & 218  & 0.84          & 288  \\ \hline
        G56                         & 4012                     & 0.77                       & 1.25                       & 146                     & \textbf{0.62}             & 0.88                    & 138  & 2.32   & 2.72  & 159 & 7.04  & 219  & 1.69          & 261  \\ \hline
        G57                         & 3492                     & \textbf{1.03}              & 1.31                       & 141                     & 1.26                      & 1.65                    & 134  & 3.58   & 4.01  & 156 & 2.90  & 220  & 2.18          & 258  \\ \hline
        G58                         & 19263                    & \textbf{0.20}              & 0.32                       & 161                     & 0.26                      & 0.33                    & 154  & 1.23   & 1.32  & 162 & 1.06  & 232  & 0.98          & 288  \\ \hline
        G59                         & 6078                     & 0.67                       & 1.07                       & 161                     & \textbf{0.64}             & 1.04                    & 154  & 1.94   & 2.33  & 177 & 5.32  & 232  & 2.62          & 303  \\ \hline
        G60                         & 14176                    & 0.35                       & 0.57                       & 202                     & \textbf{0.26}             & 0.38                    & 192  & 1.17   & 1.57  & 205 & 2.78  & 509  & 0.59          & 504  \\ \hline
        G61                         & 5789                     & 0.98                       & 1.47                       & 202                     & \textbf{0.92}             & 1.12                    & 191  & 2.68   & 3.20  & 214 & 6.64  & 509  & 2.05          & 537  \\ \hline
        G62                         & 4868                     & \textbf{1.36}              & 1.57                       & 197                     & 1.89                      & 2.11                    & 189  & 3.92   & 4.37  & 217 & 2.63  & 509  & 2.30          & 402  \\ \hline
        G63                         & 26997                    & \textbf{0.25}              & 0.33                       & 224                     & 0.30                      & 0.37                    & 214  & 0.93   & 1.51  & 242 & 0.92  & 542  & 0.89          & 576  \\ \hline
        G64                         & 8735                     & \textbf{0.45}              & 1.25                       & 225                     & 0.98                      & 1.25                    & 216  & 2.43   & 3.08  & 242 & 5.21  & 542  & 2.70          & 630  \\ \hline
        G65                         & 5558                     & \textbf{1.37}              & 1.74                       & 223                     & 1.94                      & 2.16                    & 215  & 3.82   & 4.40  & 240 & 3.04  & 465  & 2.55          & 690  \\ \hline
        G66                         & 6360                     & \textbf{1.54}              & 1.86                       & 254                     & 2.20                      & 2.43                    & 241  & 4.30   & 4.72  & 267 & 3.18  & 510  & 2.55          & 1032 \\ \hline
        G67                         & 6940                     & \textbf{1.38}              & 1.57                       & 286                     & 1.99                      & 2.19                    & 270  & 3.79   & 4.53  & 295 & 2.96  & 550  & 2.51          & 1110 \\ \hline
        G70                         & 9595                     & 0.59                       & 0.86                       & 274                     & 0.74                      & 0.87                    & 260  & 1.62   & 2.34  & 248 & 12.80 & 541  & \textbf{0.42} & 1194 \\ \hline
        G72                         & 6998                     & \textbf{1.51}              & 1.76                       & 284                     & 2.17                      & 2.40                    & 270  & 4.73   & 5.29  & 294 & 3.09  & 554  & 2.66          & 1110 \\ \hline
        G77                         & 9926                     & \textbf{1.49}              & 1.81                       & 400                     & 2.38                      & 2.58                    & 378  & 4.81   & 5.43  & 425 & 3.24  & 755  & 2.68          & 2646 \\ \hline
        G81                         & 14030                    & \textbf{1.74}              & 1.98                       & 571                     & 2.65                      & 2.80                    & 559  & 5.26   & 5.57  & 634 & 3.46  & 1038 & 2.78          & 6144 \\
        \bottomrule
    \end{tabular}
    \label{Gset std}
\end{table*}

Table \ref{Gset std} shows a detailed result on the medium-size graphs among Gset instances. We report the best gap (\texttt{best}), mean gap (\texttt{mean}) and mean time (\texttt{time}) over $20$ independent runs. Compared to the other listed algorithms, MCPG and MCPG-U achieve the maximum objective function value on all graphs except for G70. Comparing the performance of the three variants MCPG, MCPG-U and MCPG-P, we conclude that the optimized policy distribution apparently improves the performance, especially for large graphs such as G72, G77 and G81. In particular, MCPG gains a significant improvement over MCPG-U when the graph is 4-regular, such as on G32-G34.

\begin{table*}[htbp]
    \centering
    \setlength{\tabcolsep}{1.5pt}
    \caption{Detailed result on large sparse Gset instances. MCPG and its variants use edge local search as the filter function, and therefore they are suffixed by E or -E.}
    \begin{tabular}{lllllllllllllll}
        \toprule
        \multicolumn{2}{c}{Problem} & \multicolumn{3}{c}{MCPG-E} & \multicolumn{3}{c}{MCPG-UE} & \multicolumn{3}{c}{MCPG-PE} & \multicolumn{2}{c}{EO}  & \multicolumn{2}{c}{EMADM}                                                                                       \\
        \cmidrule(r){1-2} \cmidrule(r){3-5} \cmidrule(r){6-8} \cmidrule(r){9-11} \cmidrule(r){12-13} \cmidrule(r){14-15}
        name                        & UB                         & \multicolumn{2}{l}{gap}     & time                        & \multicolumn{2}{l}{gap} & time                      & \multicolumn{2}{l}{gap} & time & gap    & time  & gap  & time                       \\

                                    &                            & (best,                      & mean)                       &                         & (best,                    & mean)                   &      & (best, & mean) &      & mean  &      &      &      \\
        \cmidrule(r){1-2} \cmidrule(r){3-5} \cmidrule(r){6-8} \cmidrule(r){9-11} \cmidrule(r){12-13} \cmidrule(r){14-15}
        G55                         & 10296                      & \textbf{0.11}               & 0.34                        & 294                     & 0.22                      & 0.35                    & 288  & 1.08   & 1.45  & 490  & 2.75  & 218  & 0.84 & 288  \\ \hline
        G56                         & 4012                       & 0.67                        & 0.88                        & 297                     & \textbf{0.55}             & 0.85                    & 289  & 2.32   & 2.72  & 489  & 7.04  & 219  & 1.69 & 261  \\ \hline
        G57                         & 3492                       & \textbf{0.86}               & 1.09                        & 239                     & 0.97                      & 1.21                    & 229  & 3.58   & 4.01  & 392  & 2.90  & 220  & 2.18 & 258  \\ \hline
        G60                         & 14176                      & \textbf{0.24}               & 0.38                        & 418                     & 0.28                      & 0.36                    & 398  & 1.17   & 1.57  & 675  & 2.78  & 509  & 0.59 & 504  \\ \hline
        G61                         & 5789                       & \textbf{0.71}               & 1.00                        & 413                     & 0.83                      & 0.95                    & 396  & 2.68   & 3.20  & 675  & 6.64  & 509  & 2.05 & 537  \\ \hline
        G62                         & 4868                       & \textbf{0.99}               & 1.26                        & 330                     & 1.31                      & 1.51                    & 323  & 3.92   & 4.37  & 550  & 2.63  & 509  & 2.30 & 402  \\ \hline
        G65                         & 5558                       & \textbf{1.19}               & 1.46                        & 376                     & 1.40                      & 1.59                    & 369  & 3.82   & 4.40  & 625  & 3.04  & 465  & 2.55 & 690  \\ \hline
        G66                         & 6360                       & \textbf{1.29}               & 1.47                        & 424                     & 1.42                      & 1.86                    & 416  & 4.30   & 4.72  & 710  & 3.18  & 510  & 2.55 & 1032 \\ \hline
        G67                         & 6940                       & \textbf{1.01}               & 1.30                        & 470                     & 1.21                      & 1.55                    & 464  & 3.79   & 4.53  & 786  & 2.96  & 550  & 2.51 & 1110 \\ \hline
        G70                         & 9595                       & 0.29                        & 0.44                        & 480                     & \textbf{0.24}             & 0.42                    & 464  & 1.62   & 2.34  & 398  & 12.80 & 541  & 0.42 & 1194 \\ \hline
        G72                         & 6998                       & \textbf{1.29}               & 1.50                        & 478                     & 1.69                      & 1.84                    & 464  & 4.73   & 5.29  & 787  & 3.09  & 554  & 2.66 & 1110 \\ \hline
        G77                         & 9926                       & \textbf{1.35}               & 1.52                        & 665                     & 1.65                      & 1.86                    & 649  & 4.81   & 5.43  & 1102 & 3.24  & 755  & 2.68 & 2646 \\ \hline
        G81                         & 14030                      & \textbf{1.47}               & 1.69                        & 946                     & 1.94                      & 2.06                    & 926  & 5.26   & 5.57  & 1603 & 3.46  & 1038 & 2.78 & 6144 \\
        \bottomrule
    \end{tabular}
    \label{Gset edge}
\end{table*}

Table \ref{Gset edge} shows a detailed result on the large sparse graphs among Gset instances by using edge local search. Edge local search is a variant of local search algorithm, which flips the two nodes of a selected edges simultaneously instead of flipping one nodes. For large sparse graphs, we find that edge local search is more efficient in finding better solutions. Though edge local search leads to more time consumption for each of the sampling procedure, the number of sampling rounds is reduced to find fairly good solutions, which maintains an  acceptable time consumption. Considering the results on G77 and G81, the time consumption of EMADM  increases substantially. On the contrary, to achieve a fairly good solution, MCPG and its variants cost relatively short time on large graphs, barely proportional to the the edge size, which shows the efficiency of MCPG on particularly large graphs.

\subsubsection{Experiments on random regular graph}
To further demonstrate the effectiveness of our proposed algorithms on   large graphs, we perform the experiments on the regular graphs whose size is more than $10^4$ nodes. A $d$-regular graph is a graph where each vertex has $d$ neighbors. For various settings of degree $d$ and node number $n$, we randomly generate $5$ regular graphs using  NetworkX\cite{SciPyProceedings_11}.

Since the optimal solutions for these randomly-generated graphs are hard to obtain, we use a relative performance measure called $P$-ratio based on the results of \cite{dembo2017extremal} for $d$-regular graphs given by
$$
    P(z) = \frac{\mathrm{cut}/n - d/4}{\sqrt{d/4}}.
$$
The $P$-ratio approaches $P^* \approx 0.7632$ as the number of nodes $n \to \infty$. Therefore,  a $P$-ratio close to $P^*$ indicates that the size of the cut is close to the expected optimum.

\begin{figure}[htbp]
    \centering
    \includegraphics[width = \linewidth]{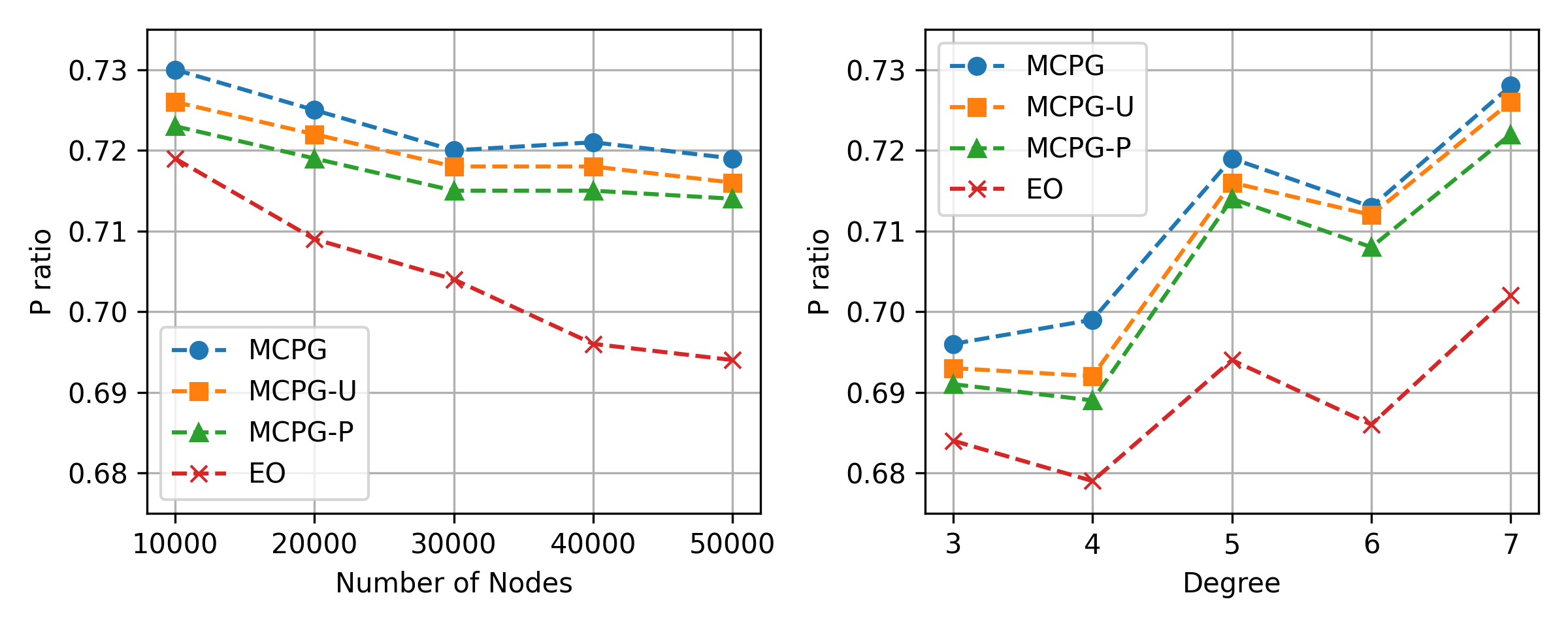}
    \caption{Computational results of regular graphs with increased number of nodes and increased number of degrees.}
    \label{fig:nodes}
\end{figure}

The tendency of the obtained cut size versus the node size and various degrees are shown in Fig. \ref{fig:nodes}. We first conduct the experiment on graphs with the same degree of $5$ and an increasing number of nodes. As the number of nodes increases, the $P$-ratio for MCPG and its variants decreases slowly. Although there remains a gap between the reported $P$ value and $P^*$, the $P$ value obtained by MCPG is far better than the results reported in \cite{toenshoff2019run}. On the contrary, as the heuristic algorithms based on local search, EO performs poorly for  large graphs. It shows that the framework of MCPG is suitable for  large graphs. 

\begin{table}[!ht]
    \centering
    \setlength{\tabcolsep}{4.5pt}
    \caption{Detailed computational results on random-generated regular graphs.}
    \begin{tabular}{llllllllll}
        \toprule
        \multicolumn{2}{c}{Problem} & \multicolumn{2}{c}{MCPG} & \multicolumn{2}{c}{MCPG-U} & \multicolumn{2}{c}{MCPG-P} & \multicolumn{2}{c}{EO}                                                 \\
        \cmidrule(r){1-2} \cmidrule(r){3-4} \cmidrule(r){5-6} \cmidrule(r){7-8} \cmidrule(r){9-10}
        node                        & d                        & obj                        & p-ratio                    & obj                    & p-ratio & obj    & p-ratio & obj    & p-ratio \\
        \cmidrule(r){1-2} \cmidrule(r){3-4} \cmidrule(r){5-6} \cmidrule(r){7-8} \cmidrule(r){9-10}
        10000                       & 5                        & \textbf{20661}             & \textbf{0.730}             & 20616                  & 0.726   & 20587  & 0.723   & 20537  & 0.719   \\ \hline
        20000                       & 5                        & \textbf{41206}             & \textbf{0.725}             & 41146                  & 0.722   & 41085  & 0.719   & 40846  & 0.709   \\ \hline
        30000                       & 5                        & \textbf{61642}             & \textbf{0.720}             & 61584                  & 0.718   & 61492  & 0.715   & 61110  & 0.704   \\ \hline
        40000                       & 5                        & \textbf{82240}             & \textbf{0.721}             & 82090                  & 0.718   & 81958  & 0.715   & 81135  & 0.696   \\ \hline
        50000                       & 3                        & \textbf{67646}             & \textbf{0.696}             & 67522                  & 0.693   & 67436  & 0.691   & 67129  & 0.684   \\ \hline
        50000                       & 4                        & \textbf{84932}             & \textbf{0.699}             & 84580                  & 0.692   & 84430  & 0.689   & 83950  & 0.679   \\ \hline
        50000                       & 5                        & \textbf{102710}            & \textbf{0.719}             & 102529                 & 0.716   & 102435 & 0.714   & 101291 & 0.694   \\ \hline
        50000                       & 6                        & \textbf{118640}            & \textbf{0.713}             & 118596                 & 0.712   & 118326 & 0.708   & 117009 & 0.686   \\ \hline
        50000                       & 7                        & \textbf{135637}            & \textbf{0.728}             & 135532                 & 0.726   & 135262 & 0.722   & 133920 & 0.702   \\
        \bottomrule
    \end{tabular}
\end{table}

Table 4 shows the detailed results on the randomly-generated graphs. we can see that when the degree is odd, our algorithms perform much better. The reason is that, when the degree is even, flipping the assignment of a node may lead to a tie of comparison in high probability. In all cases, MCPG outperforms all the other algorithms. For 4-regular graph with $50000$ nodes, the objective value achieved by MCPG exceeds MCPG-U over 300.

\subsection{The quadratic unconstrained binary optimization (QUBO)}
In this section, we demonstrate the effectiveness of our proposed algorithms on QUBO.
The goal is to solve the following problem:
\begin{equation}
    \label{prob:primary MaxCut}
    \max \quad  x^\top Q x,      \quad
    \st  \quad  x\in \{0, 1\}^n.
\end{equation}
Problem \eqref{prob:primary MaxCut} can be reformulated as \eqref{prob:intro} by converting $\{0, 1\}^n$ to $\{1,-1\}^n$ and the resulted problem still differs from  the MaxCut problem due to an extra linear term. Moreover, the sparsity of $Q$ in our experiments is greater than $0.5$, which fundamentally differs from previous Gset MaxCut instances, where the sparsity is less than $0.1$.

A large and difficult datasets called NBIQ is constructed to test our algorithms. Although there are several public available datasets (such as the Biq Mac library \cite{wiegele2007biq}), most of them are generally easy for modern solvers and computing devices, and the results on them are hard to tell which algorithm is better. Therefore, we generate the NBIQ datasets, a difficult binary quadratic optimization datasets constructed by ourselves. All the instances have a density of $0.8$, which is much denser than the instances constructed from graph data. The weights are generated randomly from uniform distribution between $10$ to $100$, and they possibly turn to be negative with a given probability. By introducing the negative weights, instances of NBIQ is significantly more difficult to be solved than existing datasets. The number of variables in all these instances is more than $5000$. 

We run MCPG and its variants $20$ times independently on each of the instances.
Similar to the MaxCut problem, we use EMADM, an SDP solver specially designed for binary quadratic programming, as a comparison algorithm. We report the gap to the best obtained objective function value. Denoting $\mathrm{UB}$ as the best results obtained by all the listed algorithm and $\mathrm{obj}$ as the cut size, the gap reported is defined as follows:
$$
    \mathrm{gap} = \frac{\mathrm{UB} - \mathrm{obj}}{\mathrm{UB}} \times 100\%.
$$

\begin{table}[!ht]
    \centering
    \setlength{\tabcolsep}{3.5pt}
    \caption{Detailed result for quadratic unconstrained binary optimization.}
    \begin{tabular}{llllllllllll}
        \toprule
        Problem      & \multicolumn{3}{c}{MCPG} & \multicolumn{3}{c}{MCPG-U} & \multicolumn{3}{c}{MCPG-P} & \multicolumn{2}{c}{EMADM}                                                      \\
        \cmidrule(r){1-1} \cmidrule(r){2-4} \cmidrule(r){5-7} \cmidrule(r){8-10} \cmidrule(r){11-12}
        Name         & gap                      &                            & time                       & gap                       &       & time & gap    &       & time & gap  & time \\
                     & (best,                   & mean)                      &                            & (best,                    & mean) &      & (best, & mean) &      &      &      \\
        \cmidrule(r){1-1} \cmidrule(r){2-4} \cmidrule(r){5-7} \cmidrule(r){8-10} \cmidrule(r){11-12}
        nbiq.5000.1  & 0.00                     & \textbf{0.42}              & 364                        & 0.24                      & 0.45  & 361  & 0.58   & 0.63  & 364  & 1.33 & 356  \\ \hline
        nbiq.5000.2  & 0.00                     & \textbf{0.50}              & 368                        & 0.00                      & 0.52  & 364  & 0.93   & 1.08  & 365  & 1.45 & 361  \\ \hline
        nbiq.5000.3  & 0.00                     & \textbf{0.56}              & 370                        & 0.29                      & 0.68  & 369  & 0.75   & 0.97  & 378  & 1.23 & 357  \\ \hline
        nbiq.5000.4  & 0.00                     & \textbf{0.69}              & 370                        & 0.00                      & 0.97  & 367  & 1.82   & 1.88  & 383  & 1.39 & 362  \\ \hline
        nbiq.5000.5  & 0.00                     & \textbf{0.74}              & 374                        & 0.17                      & 0.89  & 371  & 1.85   & 1.92  & 376  & 1.41 & 365  \\ \hline
        nbiq.7000.1  & 0.00                     & \textbf{0.39}              & 513                        & 0.00                      & 0.43  & 510  & 0.54   & 0.60  & 512  & 1.38 & 1132 \\ \hline
        nbiq.7000.2  & 0.00                     & \textbf{0.44}              & 509                        & 0.00                      & 0.50  & 507  & 0.64   & 0.75  & 510  & 1.27 & 1147 \\ \hline
        nbiq.7000.3  & 0.00                     & \textbf{0.61}              & 515                        & 0.34                      & 0.90  & 510  & 1.15   & 1.24  & 512  & 1.47 & 1139 \\ \hline
        nbiq.7000.4  & 0.00                     & \textbf{0.74}              & 519                        & 0.00                      & 0.89  & 515  & 1.46   & 1.55  & 514  & 1.25 & 1206 \\ \hline
        nbiq.7000.5  & 0.00                     & \textbf{0.92}              & 524                        & 0.32                      & 1.00  & 522  & 2.26   & 2.40  & 522  & 1.07 & 1078 \\ \hline
        nbiq.10000.1 & 0.00                     & \textbf{0.29}              & 784                        & 0.00                      & 0.35  & 782  & 0.28   & 0.37  & 783  & -    & -    \\ \hline
        nbiq.10000.2 & 0.00                     & \textbf{0.36}              & 785                        & 0.00                      & 0.48  & 784  & 0.68   & 0.71  & 781  & -    & -    \\ \hline
        nbiq.10000.3 & 0.00                     & \textbf{0.55}              & 788                        & 0.23                      & 0.56  & 784  & 0.64   & 0.91  & 783  & -    & -    \\ \hline
        nbiq.10000.4 & 0.00                     & \textbf{0.61}              & 789                        & 0.00                      & 0.89  & 785  & 0.71   & 1.08  & 787  & -    & -    \\ \hline
        nbiq.10000.5 & 0.00                     & \textbf{0.75}              & 794                        & 0.00                      & 0.98  & 789  & 1.62   & 1.91  & 793  & -    & -    \\
        \bottomrule
    \end{tabular}
    \label{biq}
\end{table}

The results on NBIQ datasets are shown in Table \ref{biq}. MCPG achieves the best objective function value in all cases. In a few cases, MCPG-U  also gets the maximum objective function value  obtained by MCPG. We can see that MCPG easily achieves the best known result, and it takes shorter time than EMADM on large instances. It should be noticed that the computing time of MCPG to reach the best known objective values is almost proportional to the number of nodes, which differs from EMADM and other existing algorithms.

\subsection{Cheeger cut}
In this subsection, we demonstrate the effectiveness of our proposed algorithms on finding the Cheeger cut. Cheeger cut is a kind of balanced graph cut, which are widely used in classification tasks and clustering. Given a graph $G = (V, E, w)$, the ratio Cheeger cut (RCC) and the normal Cheeger cut (NCC) are defined as
$$
    \begin{aligned}
        \mathrm{RCC}(S, S^c)  = \frac{\mathrm{cut}(S,S^c)}{\min\{|S|, |S^c|\}},   \quad
        \mathrm{NCC}(S, S^c)  = \frac{\mathrm{cut}(S,S^c)}{|S|} + \frac{\mathrm{cut}(S,S^c)}{|S^c|},
    \end{aligned}
$$
where $S$ is a subset of $V$ and $S^c$ is its complementary set. The task is to find the minimal ratio Cheeger cut or normal Cheeger cut, which can be converted into the following binary unconstrained programming:
$$
    \begin{aligned}
        \min & \quad \frac{\sum_{(i,j)\in E}  (1-x_ix_j)}{\min\left\{ \sum_i^n (1 + x_i), \sum_i^n (1 - x_i)\right\}}, \\
        \st  & \quad x\in\{-1,1\}^n,                                                                                   \\
    \end{aligned}$$
and
$$ \begin{aligned}
        \min & \quad \frac{\sum_{(i,j)\in E}  (1-x_ix_j)}{ \sum_i^n (1 + x_i)} + \frac{\sum_{(i,j)\in E}  (1-x_ix_j)}{\sum_i^n (1 - x_i)}, \\
        \st  & \quad x\in\{-1,1\}^n.                                                                                                       \\
    \end{aligned}
$$

There are fundamental differences between the Cheeger cut problem and MaxCut problems, since the objective function of the Cheeger cut is not differentiable, which poses a challenge to algorithms based on relaxation and rounding approximations. Therefore, the framework used for many algorithms in MaxCut is not applicable to the Cheeger cut problem.

Here we conduct the experiments on the Gset instances, and compare our algorithm with p-spectral clustering (pSC) algorithm \cite{BueHei2009}. pSC regards the second eigenvector of the graph p-Laplacian as the approximation of the minimal Cheeger cut. We modify the primal pSC algorithm by adjusting the parameter $p$ during the iteration, which accelerates the algorithm by over $3$ times. We run MCPG and its variants $10$ times independently on each of the instances. We report the RCC and NCC of the partition obtained by each of the algorthms.

Table \ref{RCC} presents the results obtained for finding the minimal RCC. Among all instances, MCPG consistently achieves the best objective function value. Notably, for instances G35, G36, and G37, MCPG outperforms pSC by reducing the RCC value by 0.7. Additionally, the execution time is observed to be approximately proportional to the number of graph nodes. In contrast, while pSC demonstrates moderate execution time for small graphs, it exhibits significant time growth as the graph size increases, rendering it impractical for larger instances. Specifically, the running time of pSC on large graph is more than 1.5 times longer compared to MCPG.

\begin{table}[htbp]
    \centering
    \setlength{\tabcolsep}{3.5pt}
    \caption{Detailed result for obtaining ratio Cheeger cut (RCC).}
    \begin{tabular}{llllllllllll}
        \toprule
        Problem & \multicolumn{3}{c}{MCPG} & \multicolumn{3}{c}{MCPG-U} & \multicolumn{3}{c}{MCPG-P} & \multicolumn{2}{c}{pSC}                                                                         \\
        \cmidrule(r){1-1} \cmidrule(r){2-4} \cmidrule(r){5-7} \cmidrule(r){8-10} \cmidrule(r){11-12}
        Name    & RCC                      &                            & time                       & RCC                     &                & time & RCC    &                & time & RCC   & time \\
                & (best,                   & mean)                      &                            & (best,                  & mean)          &      & (best, & mean)          &      &       &      \\
        \cmidrule(r){1-1} \cmidrule(r){2-4} \cmidrule(r){5-7} \cmidrule(r){8-10} \cmidrule(r){11-12}
        G35     & 2.931                    & \textbf{3.039}             & 66                         & 3.036                   & 3.163          & 65   & 3.113  & 3.216          & 65   & 3.864 & 127  \\ \hline
        G36     & 2.858                    & \textbf{2.894}             & 67                         & 2.922                   & 2.982          & 65   & 2.972  & 3.004          & 68   & 3.794 & 131  \\ \hline
        G37     & 2.847                    & \textbf{2.899}             & 68                         & 2.984                   & 3.130          & 69   & 3.050  & 3.173          & 67   & 3.895 & 134  \\ \hline
        G38     & 2.835                    & \textbf{2.861}             & 67                         & 2.875                   & 2.897          & 69   & 2.913  & 2.927          & 68   & 3.544 & 125  \\ \hline
        G48     & 0.084                    & \textbf{0.085}             & 93                         & 0.085                   & 0.088          & 93   & 0.085  & 0.085          & 94   & 0.109 & 184  \\ \hline
        G49     & 0.151                    & \textbf{0.158}             & 96                         & 0.157                   & 0.159          & 96   & 0.163  & 0.168          & 94   & 0.188 & 145  \\ \hline
        G50     & 0.033                    & 0.034                      & 93                         & 0.033                   & \textbf{0.033} & 95   & 0.034  & 0.035          & 92   & 0.040 & 140  \\ \hline
        G51     & 2.890                    & \textbf{2.908}             & 33                         & 2.914                   & 2.960          & 34   & 2.996  & 3.137          & 31   & 3.997 & 52   \\ \hline
        G52     & 2.970                    & \textbf{2.990}             & 35                         & 3.083                   & 3.149          & 37   & 3.220  & 3.364          & 33   & 3.993 & 53   \\ \hline
        G53     & 2.827                    & \textbf{2.846}             & 33                         & 2.892                   & 2.896          & 31   & 2.910  & 2.980          & 34   & 3.441 & 54   \\ \hline
        G54     & 2.852                    & \textbf{2.918}             & 34                         & 2.885                   & 3.018          & 32   & 2.928  & 3.016          & 35   & 3.548 & 57   \\ \hline
        G60     & 0.000                    & \textbf{0.000}             & 245                        & 0.000                   & \textbf{0.000} & 243  & 0.000  & \textbf{0.000} & 246  & 3.240 & 410  \\ \hline
        G63     & 3.107                    & \textbf{3.164}             & 242                        & 3.233                   & 3.358          & 244  & 3.334  & 3.481          & 241  & 4.090 & 373  \\ \hline
        G70     & 0.000                    & \textbf{0.000}             & 342                        & 0.000                   & \textbf{0.000} & 342  & 0.000  & \textbf{0.000} & 342  & 3.660 & 570  \\ \hline
        \bottomrule
    \end{tabular}
    \label{RCC}
\end{table}

Table \ref{NCC} displays the results obtained for finding the minimal NCC. MCPG consistently achieves the best objective function value across all instances. For instances G51-G54, MCPG achieves a reduction in NCC of $1.2$ compared to the results obtained by pSC. The analysis of consuming time for both MCPG and pSC in the context of NCC aligns with the conclusions drawn for RCC.

\begin{table}[htbp]
    \centering
    \setlength{\tabcolsep}{3.5pt}
    \caption{Detailed result for obtaining normal Cheeger cut (NCC).}
    \begin{tabular}{llllllllllll}
        \toprule
        Problem & \multicolumn{3}{c}{MCPG} & \multicolumn{3}{c}{MCPG-U} & \multicolumn{3}{c}{MCPG-P} & \multicolumn{2}{c}{pSC}                                                                         \\
        \cmidrule(r){1-1} \cmidrule(r){2-4} \cmidrule(r){5-7} \cmidrule(r){8-10} \cmidrule(r){11-12}
        Name    & NCC                      &                            & time                       & NCC                     &                & time & NCC    &                & time & NCC   & time \\
                & (best,                   & mean)                      &                            & (best,                  & mean)          &      & (best, & mean)          &      &       &      \\
        \cmidrule(r){1-1} \cmidrule(r){2-4} \cmidrule(r){5-7} \cmidrule(r){8-10} \cmidrule(r){11-12}
        G35     & 4.002                    & \textbf{4.002}             & 67                         & 4.002                   & \textbf{4.002} & 66   & 4.002  & \textbf{4.002} & 66   & 4.403 & 131  \\ \hline
        G36     & 4.002                    & \textbf{4.002}             & 69                         & 4.002                   & \textbf{4.002} & 70   & 4.002  & \textbf{4.002} & 70   & 5.110 & 136  \\ \hline
        G37     & 4.002                    & \textbf{4.002}             & 71                         & 4.002                   & \textbf{4.002} & 73   & 4.002  & \textbf{4.002} & 71   & 4.978 & 137  \\ \hline
        G38     & 4.002                    & \textbf{4.002}             & 68                         & 4.002                   & \textbf{4.002} & 67   & 4.002  & \textbf{4.002} & 68   & 4.448 & 131  \\ \hline
        G48     & 0.267                    & \textbf{0.272}             & 96                         & 0.278                   & 0.291          & 96   & 0.274  & 0.294          & 94   & 0.377 & 158  \\ \hline
        G49     & 0.088                    & \textbf{0.091}             & 96                         & 0.088                   & 0.093          & 96   & 0.091  & 0.097          & 94   & 0.113 & 174  \\ \hline
        G50     & 0.072                    & \textbf{0.075}             & 96                         & 0.072                   & 0.077          & 98   & 0.074  & 0.078          & 94   & 0.093 & 175  \\ \hline
        G51     & 4.378                    & \textbf{4.390}             & 34                         & 4.523                   & 4.606          & 33   & 4.575  & 4.656          & 35   & 5.689 & 62   \\ \hline
        G52     & 4.170                    & \textbf{4.260}             & 36                         & 4.216                   & 4.396          & 36   & 4.365  & 4.651          & 37   & 5.934 & 62   \\ \hline
        G53     & 4.171                    & \textbf{4.193}             & 33                         & 4.210                   & 4.341          & 33   & 4.215  & 4.500          & 35   & 5.066 & 60   \\ \hline
        G54     & 4.267                    & \textbf{4.294}             & 35                         & 4.270                   & 4.394          & 33   & 4.272  & 4.318          & 35   & 5.494 & 65   \\ \hline
        G60     & 0.000                    & \textbf{0.000}             & 255                        & 0.000                   & \textbf{0.000} & 255  & 0.000  & \textbf{0.000} & 255  & 2.150 & 467  \\ \hline
        G63     & 4.001                    & \textbf{4.001}             & 253                        & 4.001                   & \textbf{4.001} & 252  & 4.001  & \textbf{4.001} & 252  & 4.979 & 431  \\ \hline
        G70     & 0.000                    & \textbf{0.000}             & 347                        & 0.000                   & \textbf{0.000} & 345  & 0.000  & \textbf{0.000} & 347  & 4.090 & 636  \\ \hline
        \bottomrule
    \end{tabular}
    \label{NCC}
\end{table}

\subsection{Classical MIMO Detection}
We evaluate our proposed algorithm on the classical MIMO detection problems. The goal is to recover $x_C \in \mathcal Q$ from the linear model
\begin{equation}
    y_C = H_Cx_C+\nu_C,
\end{equation}
where $y_C\in \mathbb C^M$ denotes the received signal, $H_C\in \mathbb C^{M\times N}$ is the channel, $x_C$ denotes the sending signal, and $\nu_C\in \mathbb C^N\sim \mathcal N(0,\sigma^2I_N)$ is the Gaussian noise with known variance. We consider $\mathcal Q = \{z\in \mathbb C: {\rm Re}(z), {\rm Im}(z)\in\{\pm 1\}\}^N$. Our aim is to maximize the likelihood, that is equivalent to
\begin{equation}
    \min_{x_C\in\mathbb C^N}  \|H_Cx_C-y_C\|_2^2,\quad
    \st \quad x_C\in \mathcal Q .
\end{equation}

By separating the real and imaginary parts, we can reduce the problem to a binary one. Let
\begin{equation}
    \begin{gathered}
        H=
        \begin{bmatrix}
            {\rm Re}(H_C) & -{\rm Im}(H_C) \\
            {\rm Im}(H_C) & {\rm Re}(H_C)
        \end{bmatrix}
        , x=
        \begin{bmatrix}
            {\rm Re}(x_C) \\
            {\rm Im}(x_C)
        \end{bmatrix}
        , y=
        \begin{bmatrix}
            {\rm Re}(y_C) \\
            {\rm Im}(y_C)
        \end{bmatrix}.
    \end{gathered}
\end{equation}
The problem is equivalent to the following:
\begin{equation}
    \min_{x\in\mathbb R^{2N}}  \|Hx-y\|_2^2,\quad
    \st \quad x\in \{\pm1\}^{2N} .
    \label{reducedMIMO}
\end{equation}

Our methods are specifically designed to tackle the optimization problem represented by \eqref{reducedMIMO}. Considering that the differences between MCPG and its variants MCPG-U and MCPG-P are minimal, particularly when the values of $M$ and $N$ are small, we have chosen to solely focus on evaluating MCPG in this context. To establish a lower bound (LB) for the problem, we assume no interference between symbols, as described in \cite{HOTML}. However, it is important to note that the LB results are often unachievable in practice. For the purpose of comparison, we have evaluated our algorithm against HOTML, DeepHOTML, and minimum-mean-square-error (MMSE), as outlined in \cite{HOTML}. The codes for these three algorithms were directly running on a CPU workstation.

To test the performance, we examine the bit error rate (BER) performance with respect to the signal noise ratio (SNR). The BER and SNR are defined as
\begin{equation}
    {\rm BER} = \frac{{\rm card}(\{x\neq x^*\})}{2N}, \quad
    {\rm SER} = \frac{\mathbb E[\|H_Cx_C\|_2^2]}
    {\mathbb E[\|\nu_C\|_2^2]}
    = \frac{M\sigma_x^2}{\sigma_v^2} ,
\end{equation}
where $\sigma_x^2 = \mathbb E[\|x_C\|_2^2]$, $\sigma_\nu^2 = \mathbb E[\|\nu_C\|_2^2]$ are the variances, $x^*\in \{\pm 1\}^{2N}$ is the output solution and ${\rm card}(\{x\neq x^*\})$ is the cardinality of the set $\{i:x_i\neq x^*_i\}$. Thus, the noise becomes larger as the SNR decreases, and the problem becomes harder to solve.

To evaluate the performance of the algorithms on the problem, we conducted experiments using fixed matrix sizes $M$ and $N$, and varying signal-to-noise ratios (SNR) in the range ${2, 4, 6, 8, 10, 12}$. We generated datasets consisting of either 100 or 400 randomly selected benchmarks for each fixed SNR value, and calculated the average bit error rate (BER) performance of the algorithms on these benchmarks.

Specifically, we set $M=N=180$ and $M=N=200$, and each dataset contained 400 instances. The results of these experiments are presented in Table \ref{tabMIMO200} and illustrated in Fig. \ref{figMIMO200}. The ``Type" column in the table indicates the value of $M=N$ and the corresponding SNR(dB). It is important to note that in order to ensure appropriate termination of our algorithm, the parameters and execution times of MCPG varied depending on the SNR value. For each instance, MCPG took no more than 0.6 seconds to execute. The comparison of execution times among different algorithms was primarily focused on larger problem sizes. In terms of BER performance, MCPG consistently achieved the lowest BER among all algorithms on the datasets generated for each fixed SNR value. MCPG demonstrated improvement over the other three algorithms, particularly at SNR(dB) values of 2, 4, 6, and 8, with substantial improvement observed at SNR(dB) values of 10 and 12.
\begin{table*}[htbp]
    \centering
    \setlength{\tabcolsep}{3pt}
    \caption{Results on classical MIMO detection problems when $M=N=180$ and $M=N=200$.}
    \begin{tabular}{llllllllll}
        \toprule
        \multicolumn{1}{c}{Type} &
        \multicolumn{1}{c}{LB}   &
        \multicolumn{2}{c}{MCPG} & \multicolumn{2}{c}{HOTML} & \multicolumn{2}{c}{DeepHOTML} & \multicolumn{2}{c}{MMSE}                                                          \\
        \cmidrule(r){1-1} \cmidrule(r){2-2} \cmidrule(r){3-4} \cmidrule(r){5-6} \cmidrule(r){7-8} \cmidrule(r){9-10}
                                 & BER                       & BER                           & time                     & BER      & time  & BER      & time  & BER      & time  \\
        \cmidrule(r){1-1} \cmidrule(r){2-2} \cmidrule(r){3-4} \cmidrule(r){5-6} \cmidrule(r){7-8} \cmidrule(r){9-10}
        180-2                    & 0.103542                  & \textbf{0.174243}             & 0.126                    & 0.189819 & 0.238 & 0.176069 & 0.013 & 0.177729 & 0.003 \\ \hline
        180-4                    & 0.055993                  & \textbf{0.126701}             & 0.149                    & 0.143021 & 0.140 & 0.128285 & 0.004 & 0.141028 & 0.002 \\ \hline
        180-6                    & 0.022729                  & \textbf{0.070410}             & 0.303                    & 0.079924 & 0.176 & 0.077785 & 0.004 & 0.106201 & 0.002 \\ \hline
        180-8                    & 0.006014                  & \textbf{0.013438}             & 0.601                    & 0.013861 & 0.116 & 0.033701 & 0.005 & 0.075201 & 0.002 \\ \hline
        180-10                   & 0.000951                  & \textbf{0.001083}             & 0.246                    & 0.001313 & 0.081 & 0.007556 & 0.005 & 0.050181 & 0.003 \\ \hline
        180-12                   & 0.000042                  & \textbf{0.000042}             & 0.246                    & 0.000076 & 0.064 & 0.000694 & 0.006 & 0.030889 & 0.003 \\ \hline
        200-2                    & 0.105538                  & \textbf{0.174606}             & 0.169                    & 0.189975 & 0.382 & 0.180075 & 0.009 & 0.176825 & 0.003 \\ \hline
        200-4                    & 0.057081                  & \textbf{0.126763}             & 0.257                    & 0.143144 & 0.273 & 0.130656 & 0.006 & 0.140331 & 0.003 \\ \hline
        200-6                    & 0.023519                  & \textbf{0.069806}             & 0.335                    & 0.080588 & 0.294 & 0.076931 & 0.006 & 0.105919 & 0.003 \\ \hline
        200-8                    & 0.006288                  & \textbf{0.014556}             & 0.581                    & 0.015356 & 0.208 & 0.027106 & 0.006 & 0.074975 & 0.003 \\ \hline
        200-10                   & 0.000856                  & \textbf{0.001031}             & 0.423                    & 0.001438 & 0.140 & 0.003625 & 0.007 & 0.048950 & 0.004 \\ \hline
        200-12                   & 0.000019                  & \textbf{0.000025}             & 0.280                    & 0.000063 & 0.111 & 0.000100 & 0.006 & 0.030100 & 0.004 \\
        \bottomrule
    \end{tabular}
    \label{tabMIMO200}
\end{table*}

\begin{figure}[htbp]
    \centering
    \includegraphics[width = \linewidth]{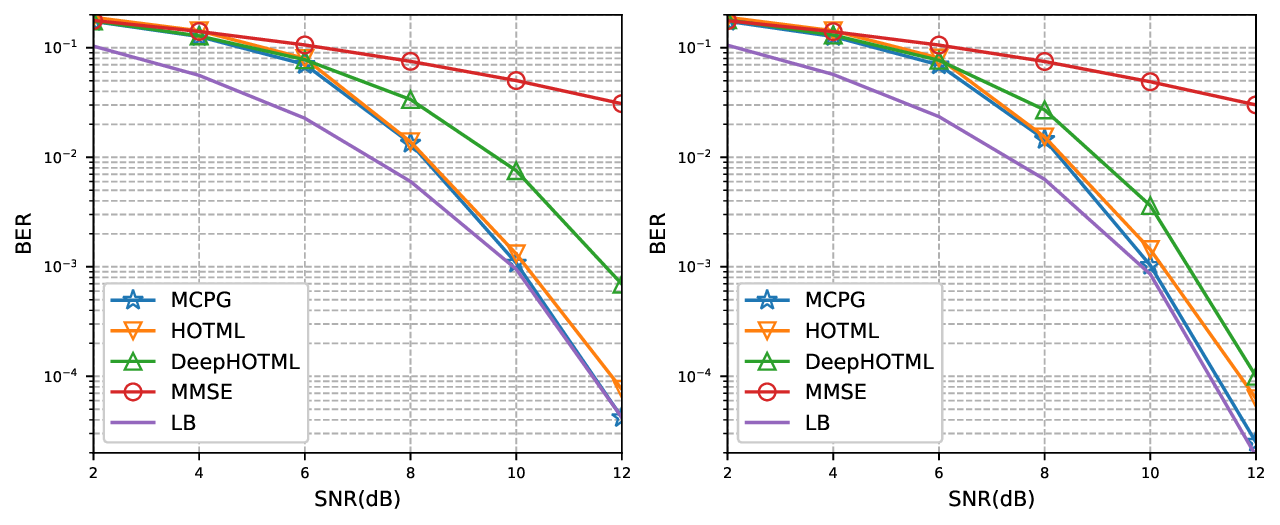}
    \begin{tabular}{p{0.5\linewidth}p{0.42\linewidth}<{\centering}}
        \centering
        (a) $M=N=180$ & (b) $M=N=200$
    \end{tabular}
    \caption{Results on classical MIMO detection problems when $M=N=180$ and $M=N=200$.}
    \label{figMIMO200}
\end{figure}

In addition to the previous experiments, we conducted further experiments on larger problem sizes, specifically $M=N=800$ and $M=N=1200$. To ensure computational efficiency, we randomly generated datasets consisting of 100 benchmarks for each case. However, due to the expensive training process of DeepHOTML, we did not run DeepHOTML for these larger cases.

The results of these experiments are presented in Table \ref{tabMIMO800} and illustrated in Fig. \ref{figMIMO800}. Among the three algorithms compared, MCPG consistently achieved the best performance in terms of BER on these larger cases as well. It is worth noting that at SNR(dB) values of 2, where the noise magnitude is large, it becomes challenging to improve the performance significantly. For SNR(dB) values of 4, 6, 10, and 12, MCPG demonstrates substantial improvements in performance compared to the other algorithms. Furthermore, it is important to note that while the time consumption of HOTML increases rapidly as the values of $M$ and $N$ become larger, the increment in time consumption of MCPG is not as significant. This suggests that MCPG is more computationally efficient compared to HOTML, particularly for larger problem sizes.

\begin{table*}[htbp]
    \centering
    \setlength{\tabcolsep}{3pt}
    \caption{Results on classical MIMO detection problems when $M=N=800$ and $M=N=1200$.}
    \begin{tabular}{llllllll}
        \toprule
        \multicolumn{1}{c}{Type} &
        \multicolumn{1}{c}{LB}   &
        \multicolumn{2}{c}{MCPG} & \multicolumn{2}{c}{HOTML} & \multicolumn{2}{c}{MMSE}                                             \\
        \cmidrule(r){1-1} \cmidrule(r){2-2} \cmidrule(r){3-4} \cmidrule(r){5-6} \cmidrule(r){7-8}
                                 & BER                       & BER                      & time & BER      & time  & BER      & time \\
        \cmidrule(r){1-1} \cmidrule(r){2-2} \cmidrule(r){3-4} \cmidrule(r){5-6} \cmidrule(r){7-8}
        800-2                    & 0.103731                  & \textbf{0.174669}        & 0.50 & 0.192981 & 10.63 & 0.177175 & 0.10 \\ \hline
        800-4                    & 0.056331                  & \textbf{0.126675}        & 1.00 & 0.146444 & 11.88 & 0.140519 & 0.10 \\ \hline
        800-6                    & 0.023131                  & \textbf{0.069094}        & 3.96 & 0.082063 & 13.47 & 0.105463 & 0.10 \\ \hline
        800-8                    & 0.006300                  & \textbf{0.012150}        & 3.29 & 0.012188 & 6.22  & 0.074900 & 0.10 \\ \hline
        800-10                   & 0.000969                  & \textbf{0.001144}        & 1.61 & 0.001363 & 3.35  & 0.049256 & 0.10 \\ \hline
        800-12                   & 0.000031                  & \textbf{0.000031}        & 1.31 & 0.000044 & 2.35  & 0.030075 & 0.09 \\ \hline
        1200-2                   & 0.104883                  & \textbf{0.174588}        & 1.00 & 0.193192 & 82.46 & 0.177675 & 0.45 \\ \hline
        1200-4                   & 0.056400                  & \textbf{0.127004}        & 1.94 & 0.145813 & 77.83 & 0.140567 & 0.47 \\ \hline
        1200-6                   & 0.023179                  & \textbf{0.070346}        & 6.47 & 0.083738 & 73.94 & 0.105979 & 0.47 \\ \hline
        1200-8                   & 0.006179                  & \textbf{0.012529}        & 7.39 & 0.012654 & 61.59 & 0.075567 & 0.47 \\ \hline
        1200-10                  & 0.000875                  & \textbf{0.001050}        & 5.03 & 0.001338 & 22.17 & 0.050167 & 0.46 \\ \hline
        1200-12                  & 0.000058                  & \textbf{0.000054}        & 3.16 & 0.000071 & 15.70 & 0.030388 & 0.46 \\
        \bottomrule
    \end{tabular}
    \label{tabMIMO800}
\end{table*}

\begin{figure}[htbp]
    \centering
    \includegraphics[width=\linewidth]{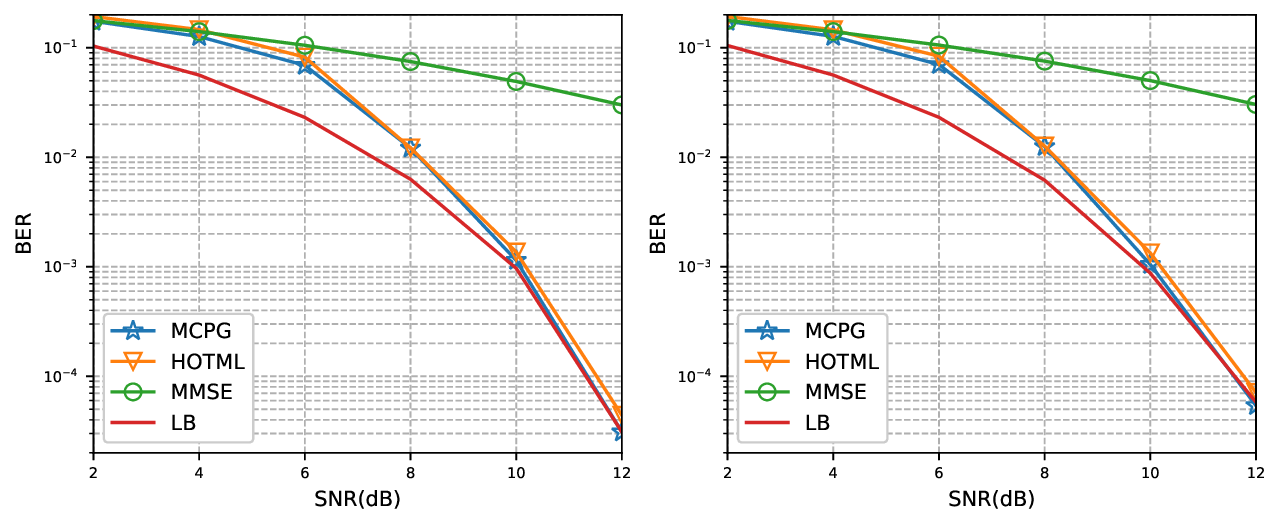}
    \begin{tabular}{p{0.5\linewidth}p{0.43\linewidth}<{\centering}}
        \centering
        (a) $M=N=800$ & (b) $M=N=1200$
    \end{tabular}
    \caption{Results on classical MIMO detection problems when $M=N=800$ and $M=N=1200$.}
    \label{figMIMO800}
\end{figure}

\subsection{MaxSAT}
In this subsection, we demonstrate the effectiveness of our proposed algorithms on (partial) MaxSAT problems.  It is a well-known combinatorial optimization problem in computer science and operations research.
Given a boolean formula in conjunctive normal form (CNF), the goal is to find an assignment of the variables that satisfies the maximum number of clauses in the formula. A formula in CNF consists of a conjunction of one or more clauses, where each clause is a disjunction of one or more literals. For example, the formula $(a \vee \neg b) \wedge (c \vee d \vee \neg e)$ is in CNF, where the first clause is $a \vee \neg b$ and the second clause is $c \vee d \vee \neg e$. Given a formula in CNF consists of clause $c^1,c^2,\cdots,c^m$, we formulate the partial MaxSAT problem as a binary programming problem as follows:
\begin{equation}
    \label{eq:maxsat}
    \begin{aligned}
        \max & \quad \sum_{c^i \in C_1} \max\{c_1^i x_1, c_2^i x_2,\cdots, c_n^i x_n , 0\}, \\ \text{s.t.} &\quad \max\{c_1^i x_1, c_2^i x_2,\cdots, c_n^i x_n, 0\} = 1, \quad \text{for } c^i \in C_2,\\
             & \quad  x \in \{-1,1\}^n,
    \end{aligned}
\end{equation}
where $C_1$ represents the soft clauses that should be satisfied as much as possible and $C_2$ represents the hard clauses that have to be satisfied. $c_j^i$ represents the sign of literal $j$ in clause $i$, i.e.,
$$
    c_j^i=\begin{cases}
        1,  & \quad  \text{$x_j$ appears in the clause $C_i$},      \\
        -1, & \quad  \text{$\neg x_j$ appears in the clause $C_i$}, \\
        0,  & \quad \text{else.}                                    \\
    \end{cases}
$$
The constraints in the partial MaxSAT problem can be converted to an exact penalty in the objective function, which is demonstrated in \eqref{eq:penaltyfun}. Since the left side of the equality constraints in \eqref{eq:maxsat} is no more than 1, the absolute function of the penalty can be dropped. Therefore, we have the following binary programming problem:
\begin{equation}
    \begin{aligned}
        \max & \quad \sum_{c^i \in C_1\cup C_2} w_i\max\{c_1^i x_1, c_2^i x_2,\cdots, c_n^i x_n , 0\}, \\ \text{s.t.} &\quad  x \in \{-1,1\}^n,
    \end{aligned}
\end{equation}
where $w_i = 1$ for $c^i \in C_1$ and $w_i = |C_1| + 1$ for $c^i \in C_2$.

We compare MCPG with its variant MCPG-U, the state-of-art incomplete solver SATLike, the state-of-art complete solver WBO and its incomplete variant WBO-inc. A complete solver tries to finds optimal solution through searching procedure and a incomplete solver usually finds a high-quality solution in the heuristic way. We do not test the learning algorithms, since they still remains poor performance compared to others. Denoting $\mathrm{UB}$ as the best results obtained by all the algorithm, and the gap reported is defined as follows:
$$
    \mathrm{gap} = \frac{\mathrm{UB} - \mathrm{obj}}{\mathrm{UB}} \times 100\%.
$$

We first conduct the experiments on the primary MaxSAT problems without hard clauses. Since the public available datasets for primary MaxSAT is too easy for modern solvers, we randomly generate a large hard dataset. The generated clauses have at most $4$ variables. To make the problem hard and close to reality, the clauses are constructed in four ways:
\begin{itemize}
    \item Clause with single variable $x_i$ or $\neg x_i$ is given. This type of clauses exists for all variable $x_i$.
    \item Two clauses $a_1\vee a_2$ and $\neg a_1\vee \neg a_2$ are given simultaneously, where $a_1$ and $a_2$ are randomly selected.
    \item Clause $a_1\vee a_2\vee a_3$ is given with three randomly selected variables $a_1,a_2,a_3$.
    \item Clauses $a_1\vee a_2\vee a_3\vee a_4$, $\neg a_1\vee \neg a_2$ and $\neg a_3\vee \neg a_4$ are given simultaneously, where the variables are randomly selected.
\end{itemize}

The results of primary MaxSAT problems are shown in Table \ref{res: MaxSAT1}. The running time of all the comparison algorithms is limited within 60, 300 or 500 seconds for instances with various sizes. MCPG outperforms all the competing algorithms, including its variant MCPG-U and the-state-of-art solvers. Not only does MCPG obtain the best results, but also the mean result of MCPG is better than the best result of other solvers. For all the three instances with 3000 variables and 10000 clauses, MCPG finds the best result every time. Note that the complete solver WBO finds the optimal solution for the two instances with 3000 variables, where MCPG also finds the optimal solution easily.

\begin{table}[htbp]
    \centering
    \caption{Results on MaxSAT  without hard clauses on  randomly generated instances. The two columns labeled "WBO/inc" in the table represent the respective results obtained from the WBO and WBO-inc algorithms. The computational time taken by the WBO and WBO-inc algorithms is the same as that of the SATlike algorithm.}
    \setlength{\tabcolsep}{2pt}
    \begin{tabular}{lllllllllllll}
        \toprule
        \multicolumn{3}{c}{Problem} & \multicolumn{3}{c}{MCPG} & \multicolumn{3}{c}{MCPG-U} & \multicolumn{2}{c}{WBO/inc} &                                                                           
        \multicolumn{2}{c}{SATLike}                                                                                                                                                                   \\
        \cmidrule(r){1-3} \cmidrule(r){4-6} \cmidrule(r){7-9} \cmidrule(r){10-11}
        \cmidrule(r){12-13}
        $n$                         & $|C_1|$                  & UB                         & gap                         &       & time & gap    &       & time & gap           & gap  & gap  & time \\
                                    &                          &                            & (best,                      & mean) &      & (best, & mean) &      &               &      &      &      \\
        \cmidrule(r){1-3} \cmidrule(r){4-6} \cmidrule(r){7-9} \cmidrule(r){10-10} \cmidrule(r){11-11}
        \cmidrule(r){12-12} \cmidrule(r){13-13}
        2000                        & 8000                     & 7211                       & \textbf{0.00}               & 0.00  & 36   & 0.04   & 0.06  & 36   & 8.17          & 5.74 & 0.08 & 60   \\\hline
        2000                        & 8000                     & 7204                       & \textbf{0.00}               & 0.01  & 36   & 0.06   & 0.08  & 35   & 8.65          & 6.18 & 0.08 & 60   \\\hline
        2000                        & 8000                     & 7211                       & \textbf{0.00}               & 0.00  & 36   & 0.01   & 0.06  & 35   & 8.68          & 6.28 & 0.07 & 60   \\\hline
        2000                        & 8000                     & 7211                       & \textbf{0.00}               & 0.02  & 35   & 0.04   & 0.06  & 35   & 8.13          & 5.96 & 0.07 & 60   \\\hline
        2000                        & 8000                     & 7200                       & \textbf{0.00}               & 0.00  & 36   & 0.06   & 0.06  & 35   & 7.89          & 5.78 & 0.07 & 60   \\\hline
        2000                        & 10000                    & 8972                       & \textbf{0.00}               & 0.01  & 39   & 0.02   & 0.03  & 38   & 6.88          & 5.49 & 0.12 & 60   \\\hline
        2000                        & 10000                    & 8945                       & \textbf{0.00}               & 0.01  & 39   & 0.02   & 0.05  & 38   & 6.47          & 5.62 & 0.15 & 60   \\\hline
        2000                        & 10000                    & 8969                       & \textbf{0.00}               & 0.01  & 39   & 0.02   & 0.04  & 38   & 7.08          & 5.74 & 0.12 & 60   \\\hline
        2000                        & 10000                    & 8950                       & \textbf{0.00}               & 0.01  & 38   & 0.02   & 0.04  & 38   & 6.74          & 5.89 & 0.13 & 60   \\\hline
        2000                        & 10000                    & 8937                       & \textbf{0.00}               & 0.01  & 39   & 0.02   & 0.03  & 38   & 6.22          & 5.99 & 0.12 & 60   \\\hline
        3000                        & 10000                    & 8954                       & \textbf{0.00}               & 0.00  & 160  & 0.02   & 0.05  & 160  & \textbf{0.00} & 5.58 & 0.01 & 300  \\\hline
        3000                        & 10000                    & 8935                       & \textbf{0.00}               & 0.00  & 161  & 0.01   & 0.04  & 161  & \textbf{0.00} & 5.60 & 0.01 & 300  \\\hline
        3000                        & 10000                    & 8926                       & \textbf{0.00}               & 0.00  & 160  & 0.01   & 0.04  & 161  & 7.78          & 5.39 & 0.01 & 300  \\\hline
        3000                        & 12000                    & 10811                      & \textbf{0.00}               & 0.00  & 166  & 0.02   & 0.04  & 165  & 8.46          & 6.26 & 0.15 & 300  \\\hline
        3000                        & 12000                    & 10823                      & \textbf{0.00}               & 0.00  & 166  & 0.01   & 0.03  & 166  & 8.20          & 5.89 & 0.12 & 300  \\\hline
        3000                        & 12000                    & 10792                      & \textbf{0.00}               & 0.00  & 165  & 0.02   & 0.03  & 166  & 8.53          & 5.93 & 0.06 & 300  \\\hline
        3000                        & 15000                    & 13712                      & \textbf{0.00}               & 0.00  & 186  & 0.01   & 0.01  & 185  & 6.66          & 5.49 & 0.13 & 300  \\\hline
        3000                        & 15000                    & 13705                      & \textbf{0.00}               & 0.00  & 186  & 0.00   & 0.01  & 185  & 5.72          & 5.44 & 0.16 & 300  \\\hline
        5000                        & 20000                    & 18032                      & \textbf{0.00}               & 0.01  & 341  & 0.05   & 0.06  & 344  & 7.83          & 6.32 & 0.11 & 500  \\\hline
        5000                        & 20000                    & 18008                      & \textbf{0.00}               & 0.00  & 344  & 0.04   & 0.06  & 342  & 7.16          & 6.26 & 0.11 & 500  \\
        \bottomrule
    \end{tabular}
    \label{res: MaxSAT1}
\end{table}

We conducted an evaluation of MCPG and other algorithms using the random track of unweighted partial MaxSAT instances from the MSE 2016 competition. The running time for all compared algorithms was limited to 60 seconds to achieve the best possible performance. For MCPG, we independently repeated the experiments 20 times on each instance and recorded the best results obtained as well as the mean gaps.

A summary of the experiments is shown in Table \ref{res: MaxSATpar}. It provides the number of the instances in which each algorithm successfully achieved the best results, i.e., the gap is zero. Notably, MCPG consistently demonstrated superior performance across various problem sets compared to SATlike. In particular, on the \texttt{min2sat} problem set, MCPG outperformed SATlike by obtaining the best result on 53 instances. Furthermore, MCPG achieved a mean gap of zero for 22 instances of \texttt{min2sat}, 33 instances of \texttt{min3sat}, and all other instances, indicating that MCPG achieved the best results consistently across the 20 repeated tests.

\begin{table}[htbp]
    \centering
    \caption{Statistics of “gap\%” on the random track datasets in MSE2016. \texttt{pct} denotes the ratio of \texttt{num} to the total number of the problem set. }
    \setlength{\tabcolsep}{2pt}
    \renewcommand\arraystretch{1.2}
    \begin{tabular}{|c|c|ll|ll|ll|ll|ll|}
        \hline
        \multirow{3}{*}{Problem Set}       & \multirow{3}{*}{Range} & \multicolumn{4}{c|}{MCPG} & \multicolumn{2}{c|}{\multirow{2}{*}{SATlike}} & \multicolumn{2}{c|}{\multirow{2}{*}{WBO}} & \multicolumn{2}{c|}{\multirow{2}{*}{WBO-inc}}                                        \\ \cline{3-6}
                                           &                        & \multicolumn{2}{c|}{best} & \multicolumn{2}{c|}{mean}                     &                                           &                                               &     &      &     &                   \\ \cline{3-12}
                                           &                        & num                       & pct                                           & num                                       & pct                                           & num & pct  & num & pct  & num & pct  \\ \hline
        \multirow{5}{*}{\texttt{min2sat}}  & 0.00                   & 53                        & 0.88                                          & 22                                        & 0.37                                          & 18  & 0.30 & 2   & 0.03 & 1   & 0.02 \\ \cline{2-12}
                                           & (0.00,1.00]            & 7                         & 0.12                                          & 38                                        & 0.63                                          & 31  & 0.52 & 0   & 0.00 & 8   & 0.13 \\ \cline{2-12}
                                           & (1.00,2.00]            & 0                         & 0.00                                          & 0                                         & 0.00                                          & 11  & 0.18 & 0   & 0.00 & 27  & 0.45 \\ \cline{2-12}
                                           & (2.00,3.00]            & 0                         & 0.00                                          & 0                                         & 0.00                                          & 0   & 0.00 & 0   & 0.00 & 19  & 0.32 \\ \cline{2-12}
                                           & $>3.00$                & 0                         & 0.00                                          & 0                                         & 0.00                                          & 0   & 0.00 & 58  & 0.97 & 5   & 0.08 \\ \hline
        \multirow{5}{*}{\texttt{min3sat}}  & 0.00                   & 50                        & 0.83                                          & 33                                        & 0.55                                          & 50  & 0.83 & 0   & 0.00 & 1   & 0.02 \\ \cline{2-12}
                                           & (0.00,1.00]            & 5                         & 0.08                                          & 16                                        & 0.27                                          & 3   & 0.05 & 0   & 0.00 & 2   & 0.03 \\ \cline{2-12}
                                           & (0.00,2.00]            & 4                         & 0.07                                          & 9                                         & 0.15                                          & 6   & 0.10 & 0   & 0.00 & 9   & 0.15 \\ \cline{2-12}
                                           & (2.00,3.00]            & 1                         & 0.02                                          & 1                                         & 0.02                                          & 1   & 0.02 & 0   & 0.00 & 8   & 0.13 \\ \cline{2-12}
                                           & $>3.00$                & 0                         & 0.00                                          & 1                                         & 0.02                                          & 0   & 0.00 & 60  & 1.00 & 40  & 0.67 \\ \hline
        \multirow{5}{*}{\texttt{pmax2sat}} & 0.00                   & 59                        & 1.00                                          & 59                                        & 1.00                                          & 59  & 1.00 & 0   & 0.00 & 0   & 0.00 \\ \cline{2-12}
                                           & (1.00,2.00]            & 0                         & 0.00                                          & 0                                         & 0.00                                          & 0   & 0.00 & 0   & 0.00 & 6   & 0.10 \\ \cline{2-12}
                                           & (2.00,3.00]            & 0                         & 0.00                                          & 0                                         & 0.00                                          & 0   & 0.00 & 2   & 0.03 & 34  & 0.58 \\ \cline{2-12}
                                           & $>3.00$                & 0                         & 0.00                                          & 0                                         & 0.00                                          & 0   & 0.00 & 57  & 0.97 & 19  & 0.32 \\ \hline
        \multirow{5}{*}{\texttt{pmax3sat}} & 0.00                   & 30                        & 1.00                                          & 30                                        & 1.00                                          & 30  & 1.00 & 0   & 0.00 & 0   & 0.00 \\ \cline{2-12}
                                           & (1.00,2.00]            & 0                         & 0.00                                          & 0                                         & 0.00                                          & 0   & 0.00 & 0   & 0.00 & 1   & 0.03 \\ \cline{2-12}
                                           & (2.00,3.00]            & 0                         & 0.00                                          & 0                                         & 0.00                                          & 0   & 0.00 & 2   & 0.07 & 13  & 0.43 \\ \cline{2-12}
                                           & $>3.00$                & 0                         & 0.00                                          & 0                                         & 0.00                                          & 0   & 0.00 & 28  & 0.93 & 16  & 0.43 \\ \hline
    \end{tabular}
    \label{res: MaxSATpar}
\end{table}

The specific results for selected individual instances are provided in Table \ref{res: MaxSATpar2}. It should be noted that while we controlled the running time of MCPG, other algorithms may have achieved its final results within the given time constraints. In the case of the selected min2sat results, MCPG may not have achieved the best outcome for every running time, but the average performance gap of MCPG surpasses that of SATlike. This observation highlights the superior capability of MCPG in achieving better overall results within nearly half of the runnning time compared to SATlike.

\begin{table}[htbp]
    \centering
    \caption{Selected results for partial MaxSAT from the MSE 2016 competition. The two columns labeled "WBO/inc" in the table represent the respective results obtained from the WBO and WBO-inc algorithms. Both two algorithms have the same time limits as SATlike.}
    \setlength{\tabcolsep}{2pt}
    \begin{tabular}{lllllllllll}
        \toprule
        \multicolumn{4}{c}{Problem} & \multicolumn{3}{c}{MCPG} & \multicolumn{2}{c}{WBO/inc} & \multicolumn{2}{c}{SATLike}                                                                                \\
        \cmidrule(r){1-4} \cmidrule(r){5-7} \cmidrule(r){8-9} \cmidrule(r){10-11}
        name                        & $|C_2$                   & $|C_1|$                     & UB                          & gap    &               & time & gap   & gap           & gap           & time \\
                                    &                          &                             &                             & (best, & mean)         &      &       &               &                      \\
        \cmidrule(r){1-4} \cmidrule(r){5-7} \cmidrule(r){8-9} \cmidrule(r){10-11}
        min2sat-800-1               & 4013                     & 401                         & 340                         & 0.00   & \textbf{0.00} & 27   & 24.41 & 2.35          & 1.47          & 60   \\ \hline
        min2sat-800-2               & 3983                     & 401                         & 352                         & 0.00   & \textbf{0.06} & 27   & 24.43 & 0.85          & 0.85          & 60   \\\hline
        min2sat-800-3               & 3956                     & 400                         & 340                         & 0.00   & \textbf{0.03} & 27   & 22.35 & 1.76          & 0.59          & 60   \\\hline
        min2sat-800-4               & 3933                     & 398                         & 349                         & 0.00   & \textbf{0.00} & 27   & 26.36 & 2.58          & 1.72          & 60   \\\hline
        min2sat-800-5               & 3871                     & 402                         & 353                         & 0.00   & 0.42          & 27   & 20.11 & 1.70          & \textbf{0.28} & 60   \\\hline
        min2sat-1040-1              & 4248                     & 525                         & 458                         & 0.00   & \textbf{0.12} & 32   & 21.62 & 2.40          & 0.22          & 60   \\\hline
        min2sat-1040-2              & 4158                     & 528                         & 473                         & 0.00   & \textbf{0.18} & 33   & 22.83 & 1.27          & 0.21          & 60   \\\hline
        min2sat-1040-3              & 4194                     & 527                         & 473                         & 0.00   & \textbf{0.29} & 33   & 17.12 & 0.21          & 0.42          & 60   \\\hline
        min2sat-1040-4              & 4079                     & 520                         & 474                         & 0.00   & \textbf{0.14} & 33   & 18.57 & 1.69          & 0.21          & 60   \\\hline
        min2sat-1040-5              & 4184                     & 523                         & 465                         & 0.43   & 0.47          & 33   & 17.42 & 1.29          & \textbf{0.00} & 60   \\\hline
        min3sat\_v80c400\_1         & 3931                     & 150                         & 95                          & 0.00   & \textbf{0.84} & 19   & 23.16 & 6.32          & 1.05          & 60   \\ \hline
        min3sat\_v80c400\_2         & 3888                     & 149                         & 99                          & 0.00   & \textbf{0.00} & 19   & 26.26 & 2.02          & \textbf{0.00} & 60   \\ \hline
        min3sat\_v80c400\_3         & 3842                     & 141                         & 101                         & 0.00   & \textbf{0.00} & 19   & 30.69 & \textbf{0.00} & \textbf{0.00} & 60   \\ \hline
        min3sat\_v80c400\_4         & 3854                     & 149                         & 95                          & 0.00   & \textbf{0.00} & 19   & 34.74 & 4.21          & 2.11          & 60   \\ \hline
        pmax2sat\_C5000\_1          & 150                      & 4850                        & 3845                        & 0.00   & \textbf{0.00} & 19   & 3.17  & 1.85          & \textbf{0.00} & 60   \\ \hline
        pmax2sat\_C5000\_2          & 150                      & 4850                        & 3859                        & 0.00   & \textbf{0.00} & 19   & 5.05  & 2.28          & \textbf{0.00} & 60   \\ \hline
        pmax2sat\_C5000\_3          & 150                      & 4850                        & 3851                        & 0.00   & \textbf{0.00} & 19   & 4.62  & 1.90          & \textbf{0.00} & 60   \\ \hline
        pmax2sat\_C5000\_4          & 150                      & 4850                        & 3857                        & 0.00   & \textbf{0.00} & 19   & 5.16  & 2.59          & \textbf{0.00} & 60   \\ \hline
        pmax3sat\_C800\_1           & 100                      & 700                         & 680                         & 0.00   & \textbf{0.00} & 10   & 1.76  & 1.32          & \textbf{0.00} & 60   \\ \hline
        pmax3sat\_C800\_2           & 100                      & 700                         & 683                         & 0.00   & \textbf{0.00} & 9    & 4.10  & 1.76          & \textbf{0.00} & 60   \\ \hline
        pmax3sat\_C800\_3           & 100                      & 700                         & 677                         & 0.00   & \textbf{0.00} & 10   & 4.14  & 1.33          & \textbf{0.00} & 60   \\ \hline
        pmax3sat\_C800\_4           & 100                      & 700                         & 675                         & 0.00   & \textbf{0.00} & 10   & 2.37  & 1.33          & \textbf{0.00} & 60   \\ \hline
        \bottomrule
    \end{tabular}
    \label{res: MaxSATpar2}
\end{table}

\section{Conclusion}
\label{sec:conclusion}
In this paper, we provide a new insight of the policy gradient method to solve binary optimization problems. To address the binary constraint and the discrete feasible region, we construct a probabilistic model to convert the binary optimization problems into a  stochastic optimization problem. Policy gradient methods are the key to guiding the policy for searching in binary spaces. Instead of direct sampling, the MCMC sampling methods are used to preserve the locality and improve the sampling efficiency. Furthermore, we incorporate filter function, i.e., the local search algorithm, into the objective function to promote the performance. The computational results show that the proposed method can compare favorably with state-of-the-art algorithms on a variety of test problems.

There are several potential directions to be improved in the future. A better architecture of neural networks should be designed to describe the combinatorial structures.   A more efficient scheme for generalization is also important.  It is interesting  to investigate more difficult  problems, including  mixed integer linear programming and nonlinear integer problems with constraints.

\nocite{}

\bibliographystyle{siam}
\bibliography{ref}

\end{document}